\documentclass[a4wide,11pt]{amsart}
\usepackage{etex}
\usepackage{amsmath, amsthm, amscd, amssymb}
\usepackage{tikz}
\usepackage{caption}
\usepackage{subcaption}
\usepackage{soul}
\usepackage{pst-node}
\usepackage[arrow,matrix]{xy}
\usepackage{graphicx}
\usepackage{amsmath, amsthm, amscd, amssymb}
\usepackage{caption,subcaption}
\usepackage{booktabs,slashbox}

\numberwithin{equation}{section}
\theoremstyle{plain}
\newtheorem{theorem}{Theorem}[section]
\newtheorem{proposition}[theorem]{Proposition}
\newtheorem{lemma}[theorem]{Lemma}
\newtheorem{corollary}[theorem]{Corollary}
\numberwithin{equation}{section}

\theoremstyle{definition}

\newtheorem{example}[theorem]{Example}

\def\Z{\mathbb{Z}} 
\def\R{\mathbb{R}} 
\def\C{\mathbb{C}} 
\def\cF{\cF} 
\def\vx{\mathbf{x}} 
\def\va{\mathbf{a}} 
\def\bk{\mathbf{k}} 
\def\cF{\mathcal{F}}
\def\cC{\mathcal{C}} 
\def\Dih{\mathrm{Dih}} 
\def\PG{\mathcal{P}G} 
\def\cdotcup{\cdot \hspace{-8pt}\bigcup} 
\def\Dih{\mathcal{D}}

\begin{document}
\title{Toric manifolds over cyclohedra}

\author{Seonjeong Park}
\address{Osaka City University Advanced Mathematical Institute, Osaka City University, 3-3-138, Sugimoto, Sumiyoshi-gu, Osaka, Japan}
\email{seonjeong1124@gmail.com}

\subjclass[2010]{14M25, 05E18}

\keywords{toric manifolds, cyclohedron, dihedral group action}

\date{\today}
\maketitle

\begin{abstract}
    We study the action of the dihedral group on the (equivariant) cohomology of the toric manifolds associated with cycle graphs.
\end{abstract}

\section{Introduction}\label{sec:intro}
    A \emph{graph} $G$ is an ordered pair $(V,E)$, where $V$ is a set of vertices and $E$ is a set of unordered pairs of nodes, called edges. A path graph $P_{n+1}$ is a graph whose vertex set is $[n+1]:=\{1,\ldots,n,n+1\}$ and edge set is $\{(i, {i+1})\mid i=1,\ldots,n\}$. A cycle graph is a graph that consists of a single cycle through all vertices, in other words, the cycle graph $C_{n+1}$ is obtained from the path graph $P_{n+1}$ by adding the edge $(1,n+1)$.

    A \emph{graph associahedron} $\mathcal{P}G$ is a simple convex polytope whose facets correspond to the connected proper subgraphs of $G$. The notion of a graph associahedron was introduced by Carr and Devadoss~(\cite{CD2006}) motivated by the associahedron. The associahedron $As^{n}$ is the $n$-dimensional simple convex polytope in which each vertex corresponds to a way of correctly inserting opening and closing parentheses in a word of $n+2$ letters and the edges correspond to single application of the associativity rule, and it can be also constructed as the graph associahedron corresponding to the path graph $P_{n+1}$. Moreover, the permutohedron $Pe^n$, the cyclohedron $Cy^n$, and the stellohedron $St^n$ are the graph associahedra corresponding to the complete graph $K_{n+1}$, the cycle graph $C_{n+1}$, and the star graph $K_{1,n}$, respectively. They have been studied in different contexts in mathematics such as algebraic combinatorics (\cite{Sta63,CFZ02}), discrete geometry (\cite{PS15}) and so on.

    An $n$-dimensional simple convex polytope is called a \emph{Delzant polytope} if the (outward) primitive normal vectors to the facets meeting at each vertex form an integral basis of $\Z^n$. Every graph associahedron can be realized as a Delzant polytope in a canonical way; we will give the canonical construction in Section~\ref{sec:preliminaries}, also see~\cite{CPh2015} for details. Hence, by the fundamental theorem of toric geometry, there is a toric manifold associated with a graph. We denoted by $M_G$ the toric manifold associated with the graph $G$.

    The actions of a finite group on toric manifolds have been also studied by many people, especially for the symmetric group. Garsia and Stanton studied the action of the symmetric group on Stanley-Reisner rings, see~\cite{GS1984}. Note that the Stanley-Reisner ring of a Delzant polytope is isomorphic to the equivariant cohomology ring of the toric manifold over the Delzant polytope. Procesi studied the action of the Weyl group on the (equivariant) cohomology of the toric manifold associated with Weyl chambers, see~\cite{Proc1990}. Note that the Weyl group of Type~A is the symmetric group.

    On the other hand, the automorphism group of the complete graph $K_{n+1}$ is the symmetric group $S_{n+1}$ and the toric manifold $M_{K_{n+1}}$ is the toric manifold associated with Weyl chambers. In general, the automorphism group $\mathrm{Aut}(G)$ of a given graph $G$ is a subgroup of the symmetric group on $V=[n+1]$ and hence one can also study the action of $\mathrm{Aut}(G)$ on the toric manifold $M_G$, and the dihedral group $\Dih_{n+1}$ is the automorphism group of the cycle graph $C_{n+1}$.

    The purpose of this paper is to deduce an explicit formula for the representation of the dihedral group on the equivariant cohomology of the toric manifold $M_{C_{n+1}}$. It should be noted that, for a toric manifold $M$, the equivariant cohomology ring $H_T^\ast(M):=H^\ast(ET\times_T M)$ is isomorphic to $H^\ast(M)\otimes H^\ast(BT)$ as an $H^\ast(BT)$-module. Hence the explicit formula for the representation on $H_T^\ast(M_{C_{n+1}})$ gives the explicit formula for the representation on the ordinary cohomology $H^\ast(M_{C_{n+1}})$.

    This paper is organized as follows: in Section~\ref{sec:preliminaries}, we review the definitions and properties of graph associahedra. Section~\ref{sec:action_nested_sets} deals with the dihedral group action on a cyclohedron. In Section~\ref{sec:action_nested_sets}, we study the subring of $H_T^\ast(M_{C_{n+1}};\C)$ determined by the facial submanifold $M_F$, which is stable under the isotropy group of a face $F$ of $Cy^n$. In Section~\ref{sec:action_equiv_coho}, we deduce an explicit formula for the representation of the dihedral group on the equivariant cohomology of $M_{C_{n+1}}$. Section~\ref{sec:annular_noncrossing} introduces a relationship between the faces of $Cy^n$ and annular non-crossing matchings.

\section{Graph associahedra}\label{sec:preliminaries}
    In this section, we review the construction and properties of the graph associahedron $\mathcal{P}G$, the simple polytope associated with a graph $G$.

    Let $G$ be a connected graph on the vertex set $[n+1]$. For a subset $I\subset [n+1]$, we denote by $G[I]$ the subgraph of $G$ whose vertex set is $I$ and whose edge set consists of all of the edges of $G$ that have both endpoints in $I$.

    Let us review the construction of the graph associahedron $\PG$. Let $\Delta^n$ be a standard simplex whose facets are (outward) normal to the standard basis vectors $-\mathbf{e}_1,\ldots,-\mathbf{e}_n$ and the vector $\sum_{i=1}^n\mathbf{e}_i$. Then we denote by $F_i$ the facet of $\Delta^n$ that is normal to the vector $-\mathbf{e}_i$ for $1\leq i\leq n$ and $F_{n+1}$ the facet normal to the vector $\sum_{i=1}^n\mathbf{e}_i$. Then there is a one-to-one correspondence between the nonempty proper subsets of $[n+1]$ and the nonempty proper faces of $\Delta^n$. Then the graph associahedron $\PG$ is obtained from $\Delta^n$ by truncating the faces corresponding to the connected proper induced subgraphs $G[I]$ in increasing order of dimension. We denote by $F_I$ the facet of $\PG$ corresponding to the connected induced subgraph $G[I]$. The graph associahedron $\PG$ is a simple polytope of dimension $n$ and it can be realized as a Delzant polytope, where the normal vector of the facet $F_I$ is equal to the vector
    \begin{equation*}
        \left\{\begin{array}{cl}-\sum_{i\in I}\mathbf{e}_i&\text{ if }n+1\not\in I,\text{ or}\\\sum_{j\not\in I}\mathbf{e}_j &\text{ if }n+1\in I.\end{array}\right.
    \end{equation*} Hence there is a complex $n$-dimensional toric manifold associated with a connected graph $G$ on $[n+1]$, and we will denote by $M_G$ the toric manifold associated with $G$.

\begin{example}
        Consider the cycle graph $C_4$, see Figure~\ref{fig:Cy3}, the first. Then the connected proper subgraphs of $C_4$ are $$1,2,3,4,12,23,34,14,123,124,134,234.\footnote{If there is no confusion, we omit the curly braces or commas to save the space.}$$ We first truncate the vertices corresponding to the subgraphs $123$, $124$, $134$, and $234$ from $\Delta^3$, see Figure~\ref{fig:Cy3}, the third. Now we truncate the edges corresponding to the subgraphs $12$, $23$, $34$, and $14$, so that we can obtain the cyclonhedron ${Cy}^3$ in Figure~\ref{fig:Cy3}, the last.
        \begin{figure}[t]
        \begin{subfigure}[c]{0.2\textwidth}
        \begin{center}
        \begin{tikzpicture}[scale=1]
            \draw (0,0)--(1,0)--(1,1)--(0,1)--cycle;
            \fill (0,0) circle (2pt);
            \fill (1,0) circle (2pt);
            \fill (0,1) circle (2pt);
            \fill (1,1) circle (2pt);
            \draw (0,-0.3) node{$1$};
            \draw (1,-0.3) node{$2$};
            \draw (1,1.3) node{$3$};
            \draw (0,1.3) node{$4$};
        \end{tikzpicture}
        \end{center}
        \end{subfigure}
        \begin{subfigure}[c]{0.25\textwidth}
        \begin{center}
        \begin{tikzpicture}[scale=0.3]
            \fill[yellow](0,5)--(-4.33,-2.5)--(4.33,-2.5)--cycle;
            \draw (0,0)--(0,5)--(-4.33,-2.5)--(4.33,-2.5)--cycle;
            \draw (0,0)--(-4.33,-2.5);
            \draw (4.33,-2.5)--(0,5);
            \draw (0,-1.25) node{\tiny$1$};
            \draw (-1,1) node{\tiny$2$};
            \draw (1,1) node{\tiny$3$};
            \draw (0,2.5) node{\tiny$23$};
            \draw (-2.16,-1.25) node{\tiny$12$};
            \draw (2.16,-1.25) node{\tiny$13$};
            \draw (2.16,1.25) node{\tiny$34$};
            \draw (-2.16,1.25) node{\tiny$24$};
            \draw (0,-2.5) node{\tiny$14$};
            \draw (0,0) node{\tiny$123$};
            \draw (-4.33,-2.5) node{\tiny$124$};
            \draw (4.33,-2.5) node{\tiny$134$};
            \draw (0,5) node{\tiny$234$};
        \end{tikzpicture}
        \end{center}
        \end{subfigure}
        \begin{subfigure}[c]{0.25\textwidth}
        \begin{center}
        \begin{tikzpicture}[scale=0.4]
            \filldraw[fill=yellow](-.8,3.4)--(-3.36,-0.97)--(-2.52,-2.42)--(2.52,-2.42)--(3.36,-0.97)--(.8,3.4)--cycle;
            \draw (-.8,3.4)--(.8,3.4)--(0,2.02)--cycle;
            \draw (0,1)--(0.86,-.5)--(-0.86,-.5)--cycle;
            \draw (3.36,-0.97)--(2.52,-2.42)--(1.74,-1.01)--cycle;
            \draw (-3.36,-0.97)--(-2.52,-2.42)--(-1.74,-1.01)--cycle;
            \draw (0,2.02)--(0,1);
            \draw (1.74,-1.01)--(0.86,-0.5);
            \draw (-1.74,-1.01)--(-0.86,-0.5);
            \draw (0,-1.25) node{\tiny$1$};
            \draw (-1,1) node{\tiny$2$};
            \draw (1,1) node{\tiny$3$};
            \draw (0,1.5) node{\tiny$23$};
            \draw (-1.3,-.75) node{\tiny$12$};
            \draw (1.3,-.75) node{\tiny$13$};
            \draw (2.16,1.25) node{\tiny$34$};
            \draw (-2.16,1.25) node{\tiny$24$};
            \draw (0,-2.5) node{\tiny$14$};
            \draw (0,0) node{\tiny$123$};
            \draw (-2.6,-1.5) node{\tiny$124$};
            \draw (2.6,-1.5) node{\tiny$134$};
            \draw (0,3.1) node{\tiny$234$};
        \end{tikzpicture}
        \end{center}
        \end{subfigure}
        \begin{subfigure}[c]{0.25\textwidth}
        \begin{center}
        \begin{tikzpicture}[scale=0.4]
            \filldraw[fill=yellow](-.8,3.4)--(-3.36,-0.97)--(-2.688,-2.12999)--(2.70449, -2.11569)--(3.18863,-1.26288)--(.48,3.4)--cycle;
            \draw (-.8,3.4)--(0.48,3.4)--(0.48, 2.848)--(0.16, 2.296)--(-0.16, 2.296)--cycle;
            \draw (-0.172, 0.7)--(0.172, 0.7)--(0.86,-.5)--(-0.520218, -0.498956)--(-0.692218, -0.201044)--cycle;
            \draw (2.712, -0.986)--(3.18863, -1.26288)--(2.70449, -2.11569)--(2.22644, -1.83969)--(1.74,-1.01)--cycle;
            \draw (-3.36,-0.97)--(-2.688, -2.12999)--(-2.2099, -1.85566)--(-1.90839, -1.28656)--(-2.06839, -1.00944)--cycle;
            \draw (1.74,-1.01)--(0.86,-0.5);
            \draw (0.16, 2.296)--(0.172, 0.7)--(-0.172, 0.7)--(-0.16, 2.296)--cycle;
            \draw (-2.06839, -1.00944)--(-0.692218, -0.201044)--(-0.520218, -0.498956)--(-1.90839, -1.28656)--cycle;
            \draw (-2.364, -2.138)--(-2.688, -2.13);
            \draw (2.712, -0.986)--(3.18863, -1.26288)--(0.48,3.4)--(0.48, 2.848)--cycle;
            \draw (-2.2099, -1.85566)--(-2.688, -2.12999)--(2.70449, -2.11569)--(2.22644, -1.83969)--cycle;
            \draw (0.64, 3.124)--(0.48,3.4);
            \draw (0,-1.25) node{\tiny$1$};
            \draw (-1,1) node{\tiny$2$};
            \draw (1,1) node{\tiny$3$};
            \draw (0,1.5) node{\tiny$23$};
            \draw (-1.3,-.75) node{\tiny$12$};
            \draw (1.8,1.25) node{\tiny$34$};
            \draw (0,-1.95) node{\tiny$14$};
            \draw (0,0) node{\tiny$123$};
            \draw (-2.6,-1.5) node{\tiny$124$};
            \draw (2.6,-1.5) node{\tiny$134$};
            \draw (0,3.1) node{\tiny$234$};
        \end{tikzpicture}
        \end{center}
        \end{subfigure}
        \caption{Construction of ${Cy}^3$}\label{fig:Cy3}
        \end{figure}
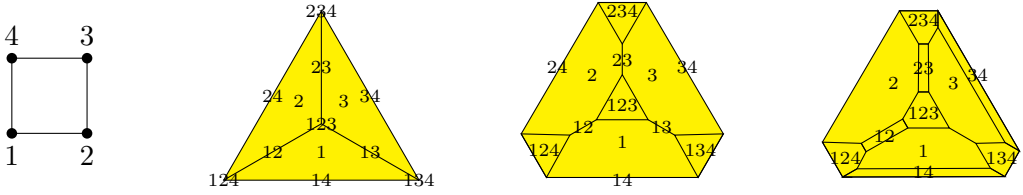
    \end{example}

    Note that two facets $F_I$ and $F_J$ of $\PG$ intersect if and only if $I\subseteq J$, $J\subseteq I$, or the induced subgraph $G[I\cup J]$ is disconnected. Hence a subset $N\subset 2^{[n+1]}\setminus [n+1]$ corresponds to a face of $\PG$ if and only if it satisfies the following three conditions.
    \begin{enumerate}
        \item[(N1)] If $I\in N$, then $G[I]$ is connected.
        \item[(N2)] If $I,J\in N$, then $I\subseteq J$, $J\subseteq I$, or $I\cap J=\emptyset$.
        \item[(N3)] For any collection of $k\geq 2$ disjoint subsets $J_1,\ldots,J_k\in N$, their union $J_1\cup\cdots\cup J_k$ does not induce a connected subgraph.
    \end{enumerate}
    A subset $N\subset 2^{[n+1]}\setminus [n+1]$ is called a \emph{nested set} of $G$ if it satisfies (N1)$\sim$(N3). Let $\mathcal{N}(G)$ be the set of nested sets of $G$, and let $\mathcal{N}_k(G)=\{N\in\mathcal{N}(G)\mid |N|=k\}$ for $0\leq k\leq n$. Then the face poset $\mathcal{F}(\PG)$ is isomorphic to the poset $\mathcal{N}(G)$ ordered by reverse inclusion, and there is one-to-one correspondence between $\mathcal{N}_k(G)$ and the set of codimension-$k$ faces of $\PG$. When $G$ is a special kind of graphs such as complete graphs, cycle graphs, path graphs, and star graphs, the face numbers of $\PG$ is well-studied. Among them, we only introduce the case when $G$ is a cycle graph.
    \begin{proposition}\cite{Simion}\label{prop:simion}
        For $k=1,\ldots,n$, the number of codimension-$k$ faces of the $n$-dimensional cyclohedron ${Cy}^n$ is equal to
        $$f_k({Cy}^n)={n\choose k}{n+i\choose k}.$$
    \end{proposition}

    Consider the polynomial ring $\bk[x_I\mid I\in\mathcal{N}_1(G)]$, where $\mathbf{k}$ is a commutative ring with unit. Then the equivariant cohomology ring of the toric manifold $M_G$, $H_T^\ast(M_G;\bk):=H^\ast(ET\times_T M_G;\bk)$, is the quotient of $\bk[x_I\mid I\in\mathcal{N}_1(G)]$ by the Stanley-Reisner ideal of $\PG$, the ideal generated by square-free monomials $x_{I_1}\cdots x_{I_k}$ for $\{I_1,\ldots,I_k\}\not\in\mathcal{N}_k(G)$. That is,
    $$H_T^\ast(M_G;\bk)=\mathbf{k}[x_I\mid I\in\mathcal{N}_1(G)]\left/\left\langle x_{I_1}\cdots x_{I_k}\mid \{I_1,\ldots,I_k\}\not\in\mathcal{N}_k(G),\,2\leq k\leq n\right\rangle\right.$$
    and for each face $F_N\in \mathcal{F}(\PG)$, the monomial $\prod_{I\in N}x_I$ is a nonzero element of $H_T^\ast(M_G)$ of degree~$2i$ for $i=1,\ldots,n$. Note that $H_T^\ast(M_G;\bk)$ is isomorphic to the Stanley-Reisner ring of $\PG$. The ordinary cohomology ring of $M_G$ is also described from the information of the graph:
    \begin{equation*}
        H^\ast(M_G;\bk)=H_T^\ast(M_G;\bk)/\langle -\sum_{\left\{i\colon i\in I,\right.\atop \left.n+1\not\in I\right\}}x_I+\sum_{\left\{i\colon i\not\in I\right.\atop \left.n+1\in I\right\}}x_I\mid i\leq i\leq n\rangle,
    \end{equation*} see~\cite{CPh2015} for details.

\section{Dihedral group action on the cyclohedron $Cy^n$}\label{sec:action_nested_sets}

    In this section, we introduce the terminologies and notations which we will use, review the properties of dihedral groups as the automorphism groups of cyclic graphs, and then we study the action of the dihedral group on the poset $\mathcal{F}(Cy^n)\cong\mathcal{N}(C_{n+1})$. 

    An automorphism of a graph $G=(V,E)$ is a permutation $\sigma$ on $V$ such that $(u,v)\in E$ if and only if $(\sigma(u),\sigma(v))\in E$. The automorphisms of $G$ form a group of $G$, and we will denote it by $\mathrm{Aut}(G)$.

    The automorphism group $\mathrm{Aut}(C_{n+1})$ is generated by a rotation and a reflection. For each positive integer $k$, the rotation $\sigma_k$ is the permutation on $[n+1]$ given by
    \begin{equation*}\label{eq:rotation}
        \sigma_k \colon [n+1]\to [n+1],~i\mapsto i+k\pmod {n+1},
    \end{equation*}
    such that $\sigma_k=(\sigma_1)^k$ and the order of $\sigma_k$ is $(n+1)/\mathrm{gcd}(k,n+1)$, where $\mathrm{gcd}(k,n+1)$ is the greatest common divisor of $k$ and $n+1$. The reflection $\tau$ is the permutation on $[n+1]$ given by
    $$\tau\colon [n+1]\to [n+1],~i\mapsto -i\pmod {n+1},$$
    such $\tau^2=e$ and $\tau\sigma_k\tau=\sigma_k^{-1}$. Then for each positive integer $k$, $\sigma_k$ and $\sigma_k\tau$ form the automorphism group of the cycle graph $C_{n+1}$, which is the dihedral group $$\Dih_{n+1}=\{\sigma_k,\,\sigma_k\tau\mid 1\leq k\leq n+1\}.$$

    Note that the face poset of $Cy^n$ is isomorphic to the poset $\mathcal{N}(C_{n+1})$ ordered by reverse inclusion. Hence to study the action of $\Dih_{n+1}$ on $Cy^n$, it is enough to see the action of $\Dih_{n+1}$ on $\mathcal{N}(C_{n+1})$.

    There is a natural action of $\Dih_{n+1}$ on $\mathcal{N}(C_{n+1})$ coming from the action of $\Dih_{n+1}$ on $C_{n+1}$; for each $\phi\in \Dih_{n+1}$, if $I=\{i_1,\ldots,i_k\}\in \mathcal{N}_1(C_{n+1})$, then $\phi\cdot I=\{\phi(i_1),\ldots,\phi(i_k)\}\in\mathcal{N}_1(C_{n+1})$, and hence if $N=\{I_1,\ldots,I_\ell\}\in\mathcal{N}_k(C_{n+1})$, then $\phi\cdot N=\{\phi\cdot I_1,\ldots,\phi\cdot I_\ell\}\in\mathcal{N}_k(C_{n+1}).$

    For each   $N\in\mathcal{N}(C_{n+1})$, we denote by $(\Dih_{n+1})_N=\{\phi\in\Dih_{n+1}\mid \phi\cdot N=N\}$, the isotropy group of $N$. If $N=\emptyset$, then $(\Dih_{n+1})_N=\Dih_{n+1}$, and otherwise, $(\Dih_{n+1})_N$ is a proper subgroup of $\Dih_{n+1}$. Note that the dihedral group $\Dih_{n+1}$ has two kinds of subgroups
    \begin{enumerate}
        \item $\langle \sigma_k\rangle$ for a divisor $k$ of $n+1$, and
        \item $\langle \sigma_k,\sigma_r\tau\rangle$ for a divisor $k$ of $n+1$ and $0\leq r<k$.
    \end{enumerate} Then the subgroup $\langle \sigma_k\rangle$ is isomorphic to the cyclic group $\cC_d$ of order $d=(n+1)/\mathrm{gcd}(k,n+1)$, and the subgroup $\langle \sigma_k,\sigma_r\tau\rangle$ is isomorphic to the dihedral group $\Dih_d$.

    For a nested set  $N=\{I_1,\ldots,I_k\}$, we set $C_{n+1}[N]:=C_{n+1}[I_1\cup\cdots\cup I_k]$, the induced subgraph of $C_{n+1}$ by the union $I_1\cup\cdots\cup I_k$. Note that for $\phi\in(\Dih_{n+1})_N$, $C_{n+1}[\phi\cdot N]=C_{n+1}[N]$.

    \begin{lemma}\label{lem:iso_nested_set}
        For each   $N\in\mathcal{N}_k(C_{n+1})$, $(\Dih_{n+1})_N$ is isomorphic to $\cC_d$ or $\Dih_d$ for some common divisor $d$ of $|N|$ and $\kappa(C_{n+1}[N])$, where $\kappa(C_{n+1}[N])$ is the number of the components of the graph $C_{n+1}[N]$.
    \end{lemma}
    \begin{proof}
        Note that $(\Dih_{n+1})_N\cong \Dih_{n+1}$ if and only if $N=\emptyset$.
        If $\kappa(C_{n+1}[N])=1$, then a nontrivial element $\phi$ fixing $N$ must be a reflection. Hence $(\Dih_{n+1})_N$ is $\langle e\rangle$ or $\Dih_1$.

        Now assume that $\kappa(C_{n+1}[N])=\ell$ and $(\Dih_{n+1})_N$ is not a subgroup of $\Dih_1$. Then $N$ can be divided into the nested sets $N_1,\ldots,N_{d}$ such that $|N_1|=\cdots=|N_{d}|$ and $(\sigma_x)^{i-1}\cdot N_1=N_{i}$ for $i=1,\ldots,d$. Then $x=\frac{n+1}{d}$ and each of $N_1,\ldots,N_d$ can be identified with each other. Hence $d$ should be a common divisor of $k$ and $\ell$. We take $d$ as big as possible. Then $\cC_d$ is a subgroup of $(\Dih_{n+1})_N$. If there is no reflection $\tau'$ in $\Dih_{n+1}$ such that $\tau'\cdot N=N$, then $(\Dih_{n+1})_N$ is the cyclic group $\langle \sigma_x\rangle\cong\cC_d$.

        If the isotropy group $(\Dih_{n+1})_N$ has also a reflection $\tau' \in \Dih_{n+1}$, then there exists an integer $i\in[d]$ such that $\tau'\cdot N_i=N_i$ or $\tau'\cdot N_i=N_{i+1}$. If $\tau'\cdot N_i=N_{i+1}$, then there exists a reflection $\tau''$ such that $\tau''\cdot N_i=N_i$. In fact, $\tau''=(\sigma_x)^{-1}\tau'$. Hence $(\Dih_{n+1})_N$ is isomorphic to $\langle\sigma_x,\tau'\rangle\cong \Dih_{d}$. Furthermore, $(\Dih_{n+1})_{N_i}\cong\Dih_1$ for each $1\leq i\leq d$.
    \end{proof}

    \begin{example}\label{ex:iso_nested}
        Consider the action of $\Dih_6$ on $\mathcal{N}(C_6)$, and the nested sets $\{12,45\}$, $\{1,4\}$, and $\{1,12,4,45\}$, see Figure~\ref{fig:dih_nested}. Then $\{12,45\}$ decomposes into two nested sets $\{12\}$ and $\{45\}$ such that $\sigma_3\cdot\{12\}=\{45\}$ and $\tau\cdot\{12\}=\{45\}$. The nested set $\{1,4\}$ also decomposes into two nested sets $\{1\}$ and $\{4\}$ such that $\sigma_3\cdot\{1\}=\{4\}$ and $(\sigma_5\tau)\cdot\{1\}=\{4\}$. Hence the nested sets $\{12,45\}$ and $\{1,4\}$ have the isotropy groups $\langle \sigma_3,\tau\rangle$ and $\langle \sigma_3,\sigma_2\tau\rangle$, respectively. Both $\langle \sigma_3,\tau\rangle$ and $\langle \sigma_3,\sigma_2\tau\rangle$ are isomorphic to $\Dih_2.$ On the other hand, there is no reflection in $\Dih_6$ preserving $\{1,12,4,45\}$, but $\{1,12,4,45\}$ decomposes into two nested sets $\{1,12\}$ and $\{4,45\}$ satisfying $\sigma_3\cdot\{1,12\}=\{4,45\}$. Hence the nested set $\{1,12,4,45\}$ has the isotropy group $\langle\sigma_3\rangle\cong \cC_2$.

        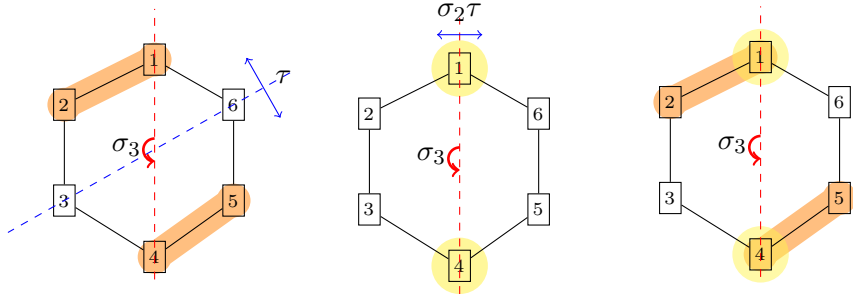
\begin{figure}[h]
        \begin{center}
            \begin{subfigure}{.25\textwidth}
            \centering
            \begin{tikzpicture}[scale=.75]
            \fill[orange!50] (4.6,-1) circle(0.3cm);
            \fill[orange!50] (4.6, 2.5) circle(0.3cm);
            \fill[orange!50] (3,1.7) circle(0.3cm);
            \fill[orange!50] (6,0) circle(0.3cm);
    	   \draw[line width=11pt, color=orange!50] (4.6,-1)--(6,0);
    	   \draw[line width=11pt, color=orange!50] (4.6, 2.5)--(3,1.7);
            \node [draw] (4) at (4.6,-1) {\!\tiny$4$\!};
        	\node [draw] (3) at (3,0) {\!\tiny$3$\!};
        	\node [draw] (5) at (6,0) {\!\tiny$5$\!};
        	\node [draw] (2) at (3,1.7) {\!\tiny$2$\!};
        	\node [draw] (6) at (6,1.7) {\!\tiny$6$\!};
            \node [draw] (1) at (4.6, 2.5) {\!\tiny$1$\!};
            \draw (1)--(2)--(3)--(4)--(5)--(6)--(1);
            \draw[blue,dashed] (7,68/30)--(2,-68/30+1.7);
            \draw[blue,shift={(6+9/17,2)},rotate=90, <->] (-9/17,-0.3)--(9/17,0.3);
            \draw (6+9/17+18/51,2.2) node{$\tau$};
            \draw[red,dashed] (4.6,-1.4)--(4.6,3.4);
            \draw [->,line width=1pt,red] (4.6,1.1) arc[x radius=0.2cm, y radius =.2cm, start angle=90, end angle=270];
            \draw (4.1,0.95) node{$\sigma_3$};
            \end{tikzpicture}
            \end{subfigure}
            \begin{subfigure}{.25\textwidth}
            \centering
            \begin{tikzpicture}[scale=.75]
            \fill[yellow!50] (4.6,-1) circle(0.5cm);
            \fill[yellow!50] (4.6, 2.5) circle(0.5cm);
            \node [draw] (4) at (4.6,-1) {\!\tiny$4$\!};
        	\node [draw] (3) at (3,0) {\!\tiny$3$\!};
        	\node [draw] (5) at (6,0) {\!\tiny$5$\!};
        	\node [draw] (2) at (3,1.7) {\!\tiny$2$\!};
        	\node [draw] (6) at (6,1.7) {\!\tiny$6$\!};
            \node [draw] (1) at (4.6, 2.5) {\!\tiny$1$\!};
            \draw (1)--(2)--(3)--(4)--(5)--(6)--(1);
            \draw[red,dashed] (4.6,-1.5)--(4.6,3.4);
            \draw [->,line width=1pt,red] (4.6,1.1) arc[x radius=0.2cm, y radius =.2cm, start angle=90, end angle=270];
            \draw (4.1,0.95) node{$\sigma_3$};
            \draw[blue,shift={(5.45,3.1)},rotate=90, <->] (0,0.45)--(0,1.25);
            \draw (4.6,3.5) node{$\sigma_2\tau$};
            \end{tikzpicture}
            \end{subfigure}
            \begin{subfigure}{.25\textwidth}
            \centering
            \begin{tikzpicture}[scale=.75]
    		\pgfsetfillopacity{0.5}
            \fill[orange] (4.6,-1) circle(0.3cm);
            \fill[orange] (4.6, 2.5) circle(0.3cm);
            \fill[orange] (3,1.7) circle(0.3cm);
            \fill[orange] (6,0) circle(0.3cm);
    	   \draw[line width=11pt, color=orange!50] (4.6,-1)--(6,0);
    	   \draw[line width=11pt, color=orange!50] (4.6, 2.5)--(3,1.7);
            \fill[yellow!80] (4.6,-1) circle(0.5cm);
            \fill[yellow!80] (4.6, 2.5) circle(0.5cm);
    		\pgfsetfillopacity{1}
            \node [draw] (4) at (4.6,-1) {\!\tiny$4$\!};
        	\node [draw] (3) at (3,0) {\!\tiny$3$\!};
        	\node [draw] (5) at (6,0) {\!\tiny$5$\!};
        	\node [draw] (2) at (3,1.7) {\!\tiny$2$\!};
        	\node [draw] (6) at (6,1.7) {\!\tiny$6$\!};
            \node [draw] (1) at (4.6, 2.5) {\!\tiny$1$\!};
            \draw (1)--(2)--(3)--(4)--(5)--(6)--(1);
            \draw[red,dashed] (4.6,-1.4)--(4.6,3.4);
            \draw [->,line width=1pt,red] (4.6,1.1) arc[x radius=0.2cm, y radius =.2cm, start angle=90, end angle=270];
            \draw (4.1,0.95) node{$\sigma_3$};
            \end{tikzpicture}
            \end{subfigure}
        \end{center}
        \caption{Isotropy groups of nested sets}\label{fig:dih_nested}
        \end{figure}
    \end{example}

    Given a cycle graph $C_{n+1}$, we define
    \begin{equation*}
        \begin{split}
            \alpha_{n+1}(d,k)&=|\{N\in\mathcal{N}_k(C_{n+1})\mid (\Dih_{n+1})_N\cong \cC_d\}|,\text{ and}\\
            \beta_{n+1}(d,k)&=|\{N\in\mathcal{N}_k(C_{n+1})\mid (\Dih_{n+1})_N\cong \Dih_d\}|.
        \end{split}
    \end{equation*}
    We can easily compute $\alpha_{n+1}(d,k)$ and $\beta_{n+1}(d,k)$ in some special cases.
    \begin{lemma} We have the following.
    \begin{enumerate}
        \item $\alpha_{n+1}(d,k)=\beta_{n+1}(d,k)=0$ if $d\nmid k$, $d\nmid n+1$, or $d+k>n+1$;
        \item $\alpha_{n+1}(1,1)=0$ and $\beta_{n+1}(1,1)=n(n+1)$;
        \item $\alpha_{n+1}(\frac{n+1}{2},\frac{n+1}{2})=0$ and $\beta_{n+1}(\frac{n+1}{2},\frac{n+1}{2})=2$ when $n+1$ is even.
    \end{enumerate}
    \end{lemma}
    \begin{proof}
        If $(\Dih_{n+1})_N$ is isomorphic to $\cC_d$ or $\Dih_d$, then it is clear that $d\mid k$ and $d\mid n+1$. Furthermore, since $k<n+1$, if $d\mid k$ and $d\mid n+1$, then we have $d+k\leq n+1$. This proves~(1).
        If $N=\{I\}$ is a singleton, then $(\Dih_{n+1})_N=\Dih_1$, hence this proves~(2).
        When $n+1$ is even, the proper maximal divisor of $n+1$ is $\frac{n+1}{2}$ and there are only two nested sets satisfying $\sigma_{2}\cdot N=N$; $\{1,3,\ldots,n\}$ and $\{2,4,\ldots,n+1\}$. This proves~(3).
    \end{proof}

    Note that $\sum_{d\mid k}\left(\alpha_{n+1}(d,k)+\beta_{n+1}(d,k)\right)$ is equal to the number of nested sets of cardinality~$k$, where the summation is taken over all divisors of $k$, hence we have
    \begin{equation}\label{eq:gamma_simion}
        \sum_{d\colon d\mid k}\left(\alpha_{n+1}(d,k)+\beta_{n+1}(d,k)\right)={ n \choose k}{n+k \choose k}
    \end{equation}
    from Proposition~\ref{prop:simion}.

    Note that each cycle graph $C_\ell$ can be presented as $\ell$ dots equally spaced on a circle; there is a one-to-one correspondence between the vertices $i\in [\ell]=V(C_\ell)$ and the dots
    $v_{\ell,i}:=\left(\cos{\frac{2\pi (i-1)}{\ell}},\sin{\frac{2\pi (i-1)}{\ell}}\right)\in S^1.$ If $\ell$ is a divisor of $n+1$, say $p\ell=n+1$, then there is a $p$-to-$1$ covering $\varphi\colon C_{n+1}\to C_\ell$ via the correspondence:
    \begin{equation*}
        \{v_{n+1,i},v_{n+1,\ell+i},\ldots,v_{n+1,(p-1)\ell+i}\} \stackrel{p:1}\longleftrightarrow \{v_{\ell,i}\}.
    \end{equation*}
    Hence if $N\in\mathcal{N}_k(C_\ell)$, then $\varphi^{-1}(N) \in \mathcal{N}_{pk}(C_{n+1})$ and $(\Dih_{n+1})_{\varphi^{-1}(N)}$ is determined by $(\Dih_\ell)_{N}$.
    If $(\Dih_\ell)_{N}\cong\cC_d$, then $(\Dih_{n+1})_{\varphi^{-1}(N)}\cong\cC_{pd}$; if $(\Dih_\ell)_{N}\cong\Dih_d$, then $(\Dih_{n+1})_{\varphi^{-1}(N)}\cong\Dih_{pd}$. One can easily see that the converse also holds.
    \begin{proposition}
        Let $\ell$ be a divisor of $n+1$, say $p\ell=n+1$. A nested set $N\in\mathcal{N}(C_{n+1})$ has the isotropy group $\Dih_{pd}$ (respectively, $\cC_{pd}$) if and only if there exists a nested set $N_0\in\mathcal{N}(C_\ell)$ such that $N=\varphi^{-1}(N_0)$ and $(\Dih_{\ell})_{N_0}\cong\Dih_d$ (respectively, $(\Dih_{\ell})_{N_0}\cong\cC_d$), where $\varphi$ is the $p$-to-$1$ covering $C_{n+1}\to C_\ell$.
    \end{proposition}
    As we saw in Example~\ref{ex:iso_nested}, the nested sets $\{12,45\}$ and $\{1,12,4,45\}$ in $\mathcal{N}(C_6)$ have the isotropy groups $\Dih_2$ and $\cC_2$, respectively. In fact, $\{12,45\}$ and $\{1,12,4,45\}$ are induced from the nested sets $\{12\}$ and $\{1,12\}$ in $\mathcal{N}(C_3)$ whose isotropy groups are $\Dih_1$ and $\langle e\rangle$, respectively.

    The proposition above tells us the following.
    \begin{corollary}
        Given a positive integer $n+1$, if $d$ is a common divisor of $n+1$ and $k$, then $$\alpha_{n+1}(d,k)=\alpha_{\frac{n+1}{d}}(1,\frac{k}{d})\text{ and } \beta_{n+1}(d,k)=\beta_{\frac{n+1}{d}}(1,\frac{k}{d}).$$ Otherwise, $\alpha_{n+1}(d,k)=\beta_{n+1}(d,k)=0$.
    \end{corollary}
    By using the M\"{o}bius inversion formula, we can compute $\gamma_{n+1}(d,k):=\alpha_{n+1}(d,k)+\beta_{n+1}(d,k)$ from Proposition~\ref{prop:simion}. We review the M\"{o}bius function and inversion formula briefly. A M\"{o}bius function for a poset $\mathcal{P}$ is a map $\mu\colon \mathcal{P}\times\mathcal{P}\to \Z$ inductively defined by the relation $$\boldsymbol{\mu}(x,y)=\begin{cases}1&\text{ for }x=y\\-\sum_{z\colon x\leq z<y}\mu(x,z)&\text{ for }x<y\\0&\text{ otherwise.}\end{cases} $$ For a finite poset $\mathcal{P}$ with M\"{o}bius function $\mu$, if $f$ and $g$ are the real valued function on $\mathcal{P}$, then the M\"{o}bius inversion formula says that the following are equivalent:$$f(x)=\sum_{y\leq x}g(y),\text{ for all }x\in \mathcal{P}$$ $$g(x)=\sum_{y\leq x}\mu(y,x)f(y)\text{ for all }x\in\mathcal{P}.$$
    \begin{lemma}\label{lem:same_size_nested_set}
        The number of nested sets $N\in\mathcal{N}_k(C_{n+1})$ such that $(\cC_{n+1})_N$ is isomorphic to $\cC_d$ or $\Dih_{d}$ is
        $$\gamma_{n+1}(d,k)=\sum_{i\mid\mathrm{gcd}(\frac{n+1}{d},\frac{k}{d}) }\mu(i){\frac{n+1}{id}-1\choose \frac{k}{id}}{\frac{n+k+1}{id}-1\choose \frac{k}{id}},$$
        where $\mu$ is the classical M\"{o}bius function\footnote{The classical M\"obius function is defined on the set of positive integers by $\mu(n)=(-1)^k$ if $n$ is the product of $k$ distinct primes and $\mu(n)=0$ if $n$ is divisible by a square.} in number theory.
    \end{lemma}
    \begin{proof}
        Consider the poset $\mathcal{P}=\{(p,q)\in \Z\times\Z\mid p>q>0\}$ ordered by
        \begin{equation*}
        \begin{split}
            (p,q)\leq (p',q') &\Leftrightarrow \text{ there exists }d\in\Z \text{ such that }p'=dp\text{ and }q'=dq.
        \end{split}
        \end{equation*} That is, $(p,q)\leq (p',q')$ if and only if $(p,q)=(\frac{p'
            }{d}, \frac{q'}{d})\text{ for some }d\mid\mathrm{gcd}(p',q')$.
        Define the integer valued functions $f$ and $g$ on $\mathcal{P}$ by $f(r,s)=\gamma_{r}(1,s)$ and $g(p,q)={p-1\choose q}{p+q-1\choose q}$, respectively. Then we can rewrite \eqref{eq:gamma_simion} as
        \begin{equation*}
            g(n+1,k)={n\choose k}{n+k\choose k}=\sum_{d\mid k}\gamma_{n+1}(d,k)=\sum_{d\mid\mathrm{gcd}(n+1,k)}\gamma_{\frac{n+1}{d}}(1,\frac{k}{d}) =\sum_{d\mid\mathrm{gcd}(n+1,k)}f(\frac{n+1}{d},\frac{k}{d}),
        \end{equation*} where the third identity comes from the fact $\gamma_{n+1}(d,k)=0$ for $d\nmid n+1$.
        From the M\"{o}bius inversion formula, we get
        \begin{equation}\label{eq:use_inversion_1}
            f(n+1,k) =\sum_{d\mid\mathrm{gcd}(n+1,k)}g(\frac{n+1}{d},\frac{k}{d}) \boldsymbol{\mu}((\frac{n+1}{d},\frac{k}{d}),(n+1,k)),
        \end{equation} where $\boldsymbol{\mu}$ is the M\"{o}bius function of the poset $\mathcal{P}$.
        Note that the closed interval $[(\frac{n+1}{d},\frac{k}{d}),(n+1,k))]$ of $\mathcal{P}$ is isomorphic to the poset $\mathcal{Q}=\{i\colon i\mid d\}$ with $i\leq_{\mathcal{Q}}i'\Leftrightarrow i\mid i'$ via the correspondence $i(\frac{n+1}{d},\frac{k}{d})\in\mathcal{P}\leftrightarrow i\in\mathcal{Q}$. Hence $\boldsymbol{\mu}((\frac{n+1}{d},\frac{k}{d}),(n+1,k))$ is equal to $\boldsymbol{\mu}(\mathcal{Q})=\mu(d)$.
        Hence from~\eqref{eq:use_inversion_1} we get the following:
        \begin{equation*}
            \gamma_{n+1}(d,k)=f(\frac{n+1}{d},\frac{k}{d}) =\sum_{i \mid \mathrm{gcd}(\frac{n+1}{d},\frac{k}{d})}g\left(\frac{n+1}{id},\frac{k}{id}\right) {\mu}(i).
        \end{equation*}
        This proves the proposition.
    \end{proof}

    Hence if we know one of $\alpha_{n+1}(d,k)$ and $\beta_{n+1}(d,k)$, then the other follows from $\gamma_{n+1}(d,k)$. We will discuss the computation of $\beta_{n+1}(d,k)$ in Section~\ref{sec:annular_noncrossing}.

\section{Hilbert series of the subrings of $H_T^\ast(M_{C_{n+1}};\C)$ determined by nested sets}\label{sec:subring_nested_set}
    In this section, we study the action of the dihedral group $\Dih_{n+1}$ on $H_T^\ast(M_{C_{n+1}};\C)$. For each nested set $N\in\mathcal{N}(C_{n+1})$, we first describe the subring of $H_T^\ast(M_{C_{n+1}};\C)$ determined by~$N$, and then compute its Hilbert series.

    For simplicity, we set
    \begin{equation*}
        \vx_N=\left\{\begin{array}{ll}
            \prod_{I\in N}x_I&\text{ for }N\in \mathcal{N}(C_{n+1})\setminus\emptyset,\\
            1 &\text{ for }N=\emptyset.
        \end{array}\right.
    \end{equation*}
    Then there is a natural action of $\Dih_{n+1}$ on $H_T^\ast(M_{C_{n+1}};\C)$; for every $\phi\in \Dih_{n+1}$ and a nested set $N\in \mathcal{N}(C_{n+1})$, $$\phi\cdot(\prod_{I\in N}x_I)=\prod_{I\in N} x_{\phi\cdot I}.$$
    Then $H_T^\ast(M_{C_{n+1}};\C)$ is isomorphic to
    \begin{equation}\label{eq:decom2_equiv_coho}
        \bigoplus_{N\in{\mathcal{N}}(C_{n+1})}\C[x_I\mid I\in N]\vx_N \cong \bigoplus_{N\in{\mathcal{N}}(C_{n+1})}\bigoplus_{\va\in(\Z_{>0})^{N}}\C\langle \prod_{I\in N}x_I^{a_I}\rangle,
    \end{equation}
    where $\va=(a_I\mid I\in N)\in (\Z_{>0})^{N}$.
    For simplicity, set $\vx_N^\va:=\prod_{I\in N}x_I^{a_I}$. Then the action of $\Dih_{n+1}$ on $H_T^\ast(M_{C_{n+1}};\C)$ is defined by $$\phi \cdot\vx_N^\va=\vx_{\phi\cdot N}^\va,\text{ that is, }\phi \cdot \prod_{I\in N}x_I^{a_I}=\prod_{I\in N}x_{\phi\cdot I}^{a_I}\text{ for every }\phi\in\Dih_{n+1}.$$ Hence $\C[x_I\mid I\in N]\vx_N$ is $(\Dih_{n+1})_N$-stable.

    Let $N$ be a nested set in $\mathcal{N}(C_{n+1})$ whose isotropy group $(\Dih_{n+1})_N$ is isomorphic to $\Dih_d$ or $\cC_d$. Without loss of generality, it is enough to consider the cases
        \begin{equation*}
            (\Dih_{n+1})_N=\begin{cases}
                \langle \sigma_{\frac{n+1}{d}}\rangle &\text{ if }(\Dih_{n+1})_N\cong\cC_d,\\
                \langle \sigma_{\frac{n+1}{d}},\tau\rangle &\text{ if }(\Dih_{n+1})_N\cong\Dih_d.
            \end{cases}
        \end{equation*}
    For simplicity, we write $\sigma_N:=\sigma_{\frac{n+1}{d}}$. Then $d\mid |N|$ and we can decompose $N$ into the nested sets $N_1,\ldots,N_d$ such that $(\sigma_N)^{i-1}\cdot N_1=N_i$. Hence when $(\Dih_{n+1})_{N}=\cC_d$, we label the elements of each $N_i$ as follows:
        \begin{equation*}
            N_i=\{A_{i,j}\mid 1\leq j\leq a\}
        \end{equation*}
    such that $(\sigma_N)^{i-1}\cdot A_{1,j}=A_{i,j}$. Note that if $(\Dih_{n+1})_N\cong\Dih_d$, then $(\Dih_{n+1})_{N_1}=\Dih_1$ and hence there is a reflection $\tau'\in (\Dih_{n+1})_N$ such that $\tau'\cdot N_1=N_1$, and some of the elements in $N$ are fixed by the reflection $\tau'$. Let $a$ be the number of elements in $N_1$ such that $\tau'\cdot I=I$. Then $|N|-a$ is even, say $2b$. Then we label the elements of each $N_i$ as
        \begin{equation}\label{eq:sets_A_B}
            N_i=\{A_{i,j}\mid 1\leq j\leq a\}\cup\{B_{i,k}\mid 1\leq k\leq 2b\}
        \end{equation}
    such that $(\sigma_N)^{i-1}\cdot A_{1,j}=A_{i,j}$, $(\sigma_N)^{i-1}\cdot B_{1,k}=B_{i,k}$, and $\tau'\cdot B_{1,k}=B_{1,2b+1-k}$ for each $i\in [d]$, $j\in [a]$, and $k\in [2b]$.

    For simplicity, for each $\vx_N^\va\in H_T^\ast(M_{C_{n+1}};\C)$ we denote by $a_{i,j}$ (respectively, $b_{i,j}$) the exponent of $x_{A_{i,j}}$ (respectively, $x_{B_{i,j}}$). That is,
    \begin{equation*}
        \vx_N^\va=\left(\prod_{1\leq i \leq d\atop 1\leq j\leq a}x_{A_{i,j}}^{a_{i,j}}\right)\left(\prod_{1\leq i\leq d \atop 1\leq j \leq 2b}x_{B_{i,j}}^{b_{i,j}}\right).
    \end{equation*} Note that the exponents $a_{i,j}$ and $b_{i,j}$ are positive integers.
    Then we can compute the isotropy group $(\Dih_{n+1})_{\vx_N^\va}$ for each $\vx_N^\va\in H_T^\ast(M_{C_{n+1}};\C)$ as follows.
    \begin{lemma}\label{lem:iso_vx_va}
        For the action of $\Dih_{n+1}$ on $H_T^\ast(M_{C_{n+1}};\C)$, the isotropy group $(\Dih_{n+1})_{\vx_N^\va}$ is a subgroup of the isotropy group $(\Dih_{n+1})_N$ with respect to the action of $\Dih_{n+1}$ on $\mathcal{N}(C_{n+1})$.
    \end{lemma}
    \begin{proof}
        Note that if $\phi\not\in(\Dih_{n+1})_N$, then $\phi\cdot N\neq N$. Hence $\phi\cdot \vx_N^\va=\vx_{\phi\cdot N}^\va\neq\vx_N^\va$.

        For a divisor $\ell$ of $d$, if $a_{i,j}=a_{i',j}$ for $i\equiv i'\pmod \ell$, then $(\sigma_N)^\ell\cdot\vx_N^\va=\vx_N^\va$. If there is a reflection $\tau'\in (\Dih_{n+1})_N$ satisfying the condition
        \begin{equation}\label{eq:iso_cyclic}
            a_I=a_{\tau'\cdot I}\text{ for any }I\in N,
        \end{equation} then $\tau'\cdot\vx_N^\va=\vx_N^\va.$

        Let $\ell$ be the smallest divisor of $d$ satisfying
        \begin{equation}\label{eq:cond_cyclic}
            a_{i,j}=a_{i',j}\text{ and }b_{i,k}=b_{i',k} \text{ for }i\equiv i'\pmod {\frac{d}{\ell}}
        \end{equation}
        for each $1\leq j\leq a$ and $1\leq k\leq 2b$.
        Then for $i=1,\ldots,\ell$, we may set
        $$L_i:=\left\{A_{p,j},\,B_{p,k}\,\middle|\, \frac{(i-1)d}{\ell}+1\leq p\leq \frac{id}{\ell},\,1\leq j\leq a,\text{ and }1\leq k\leq 2b\right\}.$$
        Then $(\sigma_N)^{\frac{d}{\ell}(i-1)}\cdot L_1=L_i$ for $1\leq i\leq {\ell}$.
        For the monomial $\vx_N^\va$ satisfying~\eqref{eq:cond_cyclic}, there exists a reflection satisfying the condition~\eqref{eq:iso_cyclic} if and only if there exists a reflection $\tau''$ such that $\tau''\cdot \vx_{L_1}^\va=\vx_{L_1}^\va$. Therefore, $(\Dih_{n+1})_{\vx_N^\va}\cong \Dih_\ell$ if and only if the exponent $\va$ satisfies~\eqref{eq:cond_cyclic} and
        \begin{equation}\label{eq:cond_reflect}
            a_{i,j}=a_{i',j}\text{ and }b_{i,k}=b_{i',k}\text{ for }i+i'=\frac{d}{\ell}+1,\,k+k'=2b+1
        \end{equation}
        for $1\leq i, i'\leq \frac{d}{\ell}$. If $\va$ satisfies~\eqref{eq:cond_cyclic} but not~\eqref{eq:cond_reflect}, then $(\Dih_{n+1})_{\vx_N^\va}\cong \cC_\ell$. This proves that $(\Dih_{n+1})_{\vx_N^\va}$ is a subgroup of $(\Dih_{n+1})_N$ in any case.
    \end{proof}

    \begin{example}\label{ex:iso_dih}
    Let us consider the cycle graph $C_{24}$ and a nested set $$N=\left\{\{i\},\{j-1,j,j+1\}\mid i\in\{1,3,...,23\}\text{ and }j\in\{2,6,10,14,18,22\}\right\}.$$ Then $(\Dih_{24})_N=\langle \sigma_4,\tau\rangle\cong\Dih_6$, see Figure~\ref{fig:isotropy_Dih}-(a).
        \begin{figure}[t]
          \centering
          \begin{subfigure}{.45\textwidth}
            \centering
            \begin{tikzpicture}[scale=.7]
                \draw (0,0) circle (2.5cm);
                \foreach \x in {1,2,...,24}
                    {
                    \filldraw[fill=white] (15*\x:2.5cm) circle(6pt);
                    \draw (15*\x-15:2.5cm) node {\tiny$\x$};
                    }
                \draw (2.5,0) node {\tiny$1$};
                \foreach \x in {105,135,...,435}
                {
                \draw[red] (\x:2.5) circle (7pt);
                }
                \draw[shift={(120:2.5cm)},rotate=30,red] (0,0) ellipse(30pt and 13pt);
                \draw[shift={(180:2.5cm)},rotate=90,red] (0,0) ellipse(30pt and 13pt);
                \draw[shift={(240:2.5cm)},rotate=150,red] (0,0) ellipse(30pt and 13pt);
                \draw[shift={(300:2.5cm)},rotate=30,red] (0,0) ellipse(30pt and 13pt);
                \draw[shift={(0:2.5cm)},rotate=90,red] (0,0) ellipse(30pt and 13pt);
                \draw[shift={(60:2.5cm)},rotate=150,red] (0,0) ellipse(30pt and 13pt);
            \end{tikzpicture}
            \subcaption{$(\Dih_{24})_N\cong\Dih_6$}
          \end{subfigure}
          \begin{subfigure}{.45\textwidth}
            \centering
            \begin{tikzpicture}[scale=.7]
                \fill[shift={(120:2.5cm)},rotate=30,blue!20] (0,0) ellipse(30pt and 13pt);
                \fill[shift={(180:2.5cm)},rotate=90,purple!20] (0,0) ellipse(30pt and 13pt);
                \fill[shift={(240:2.5cm)},rotate=150,blue!20] (0,0) ellipse(30pt and 13pt);
                \fill[shift={(300:2.5cm)},rotate=30,purple!20] (0,0) ellipse(30pt and 13pt);
                \fill[shift={(0:2.5cm)},rotate=90,blue!20] (0,0) ellipse(30pt and 13pt);
                \fill[shift={(60:2.5cm)},rotate=150,purple!20] (0,0) ellipse(30pt and 13pt);
                \draw (0,0) circle (2.5cm);
                \foreach \x in {1,2,...,24}
                    {
                    \filldraw[fill=white] (15*\x:2.5cm) circle(6pt);
                    \draw (15*\x-15:2.5cm) node {\tiny$\x$};
                    }
                \draw (2.5,0) node {\tiny$1$};
                    \fill[yellow!50] (105:2.5cm) circle(6pt);
                    \fill[yellow!50] (225:2.5cm) circle(6pt);
                    \fill[yellow!50] (345:2.5cm) circle(6pt);
                    \fill[orange!50] (135:2.5cm) circle(6pt);
                    \fill[orange!50] (255:2.5cm) circle(6pt);
                    \fill[orange!50] (375:2.5cm) circle(6pt);
                    \fill[green!50] (165:2.5cm) circle(6pt);
                    \fill[green!50] (285:2.5cm) circle(6pt);
                    \fill[green!50] (405:2.5cm) circle(6pt);
                    \fill[gray!30] (195:2.5cm) circle(6pt);
                    \fill[gray!30] (315:2.5cm) circle(6pt);
                    \fill[gray!30] (435:2.5cm) circle(6pt);
                \foreach \x in {105,135,...,435}
                {
                \draw[red] (\x:2.5) circle (7pt);
                }
                \draw[shift={(120:2.5cm)},rotate=30,red] (0,0) ellipse(30pt and 13pt);
                \draw[shift={(180:2.5cm)},rotate=90,red] (0,0) ellipse(30pt and 13pt);
                \draw[shift={(240:2.5cm)},rotate=150,red] (0,0) ellipse(30pt and 13pt);
                \draw[shift={(300:2.5cm)},rotate=30,red] (0,0) ellipse(30pt and 13pt);
                \draw[shift={(0:2.5cm)},rotate=90,red] (0,0) ellipse(30pt and 13pt);
                \draw[shift={(60:2.5cm)},rotate=150,red] (0,0) ellipse(30pt and 13pt);
            \end{tikzpicture}
            \subcaption{$\sigma_8\cdot\vx_N^\va=\vx_N^\va$}
          \end{subfigure}
          \begin{subfigure}{.45\textwidth}
            \centering
            \begin{tikzpicture}[scale=.7]
                \fill[shift={(120:2.5cm)},rotate=30,blue!20] (0,0) ellipse(30pt and 13pt);
                \fill[shift={(180:2.5cm)},rotate=90,blue!20] (0,0) ellipse(30pt and 13pt);
                \fill[shift={(240:2.5cm)},rotate=150,blue!20] (0,0) ellipse(30pt and 13pt);
                \fill[shift={(300:2.5cm)},rotate=30,blue!20] (0,0) ellipse(30pt and 13pt);
                \fill[shift={(0:2.5cm)},rotate=90,blue!20] (0,0) ellipse(30pt and 13pt);
                \fill[shift={(60:2.5cm)},rotate=150,blue!20] (0,0) ellipse(30pt and 13pt);
                \draw (0,0) circle (2.5cm);
                \foreach \x in {1,2,...,24}
                    {
                    \filldraw[fill=white] (15*\x:2.5cm) circle(6pt);
                    \draw (15*\x-15:2.5cm) node {\tiny$\x$};
                    }
                \draw (2.5,0) node {\tiny$1$};
                    \fill[yellow!50] (105:2.5cm) circle(6pt);
                    \fill[yellow!50] (435:2.5cm) circle(6pt);
                    \fill[green!50] (135:2.5cm) circle(6pt);
                    \fill[green!50] (405:2.5cm) circle(6pt);
                    \fill[green!50] (165:2.5cm) circle(6pt);
                    \fill[green!50] (375:2.5cm) circle(6pt);
                    \fill[yellow!50] (195:2.5cm) circle(6pt);
                    \fill[yellow!50] (345:2.5cm) circle(6pt);
                    \fill[yellow!50] (225:2.5cm) circle(6pt);
                    \fill[yellow!50] (315:2.5cm) circle(6pt);
                    \fill[green!50] (255:2.5cm) circle(6pt);
                    \fill[green!50] (285:2.5cm) circle(6pt);
                \foreach \x in {105,135,...,435}
                {
                \draw[red] (\x:2.5) circle (7pt);
                }
                \draw[shift={(120:2.5cm)},rotate=30,red] (0,0) ellipse(30pt and 13pt);
                \draw[shift={(180:2.5cm)},rotate=90,red] (0,0) ellipse(30pt and 13pt);
                \draw[shift={(240:2.5cm)},rotate=150,red] (0,0) ellipse(30pt and 13pt);
                \draw[shift={(300:2.5cm)},rotate=30,red] (0,0) ellipse(30pt and 13pt);
                \draw[shift={(0:2.5cm)},rotate=90,red] (0,0) ellipse(30pt and 13pt);
                \draw[shift={(60:2.5cm)},rotate=150,red] (0,0) ellipse(30pt and 13pt);
            \end{tikzpicture}
            \subcaption{$\sigma_8\cdot\vx_N^{\va}=\tau\cdot\vx_N^{\va}=\vx_N^{\va}$}
          \end{subfigure}
          \begin{subfigure}{.45\textwidth}
            \centering
            \begin{tikzpicture}[scale=.7]
                \fill[shift={(120:2.5cm)},rotate=30,blue!20] (0,0) ellipse(30pt and 13pt);
                \fill[shift={(180:2.5cm)},rotate=90,purple!20] (0,0) ellipse(30pt and 13pt);
                \fill[shift={(240:2.5cm)},rotate=150,brown!50] (0,0) ellipse(30pt and 13pt);
                \fill[shift={(300:2.5cm)},rotate=30,brown!50] (0,0) ellipse(30pt and 13pt);
                \fill[shift={(0:2.5cm)},rotate=90,purple!20] (0,0) ellipse(30pt and 13pt);
                \fill[shift={(60:2.5cm)},rotate=150,blue!20] (0,0) ellipse(30pt and 13pt);
                \draw (0,0) circle (2.5cm);
                \foreach \x in {1,2,...,24}
                    {
                    \filldraw[fill=white] (15*\x:2.5cm) circle(6pt);
                    \draw (15*\x-15:2.5cm) node {\tiny$\x$};
                    }
                \draw (2.5,0) node {\tiny$1$};
                    \fill[yellow!50] (105:2.5cm) circle(6pt);
                    \fill[yellow!50] (435:2.5cm) circle(6pt);
                    \fill[green!50] (135:2.5cm) circle(6pt);
                    \fill[green!50] (405:2.5cm) circle(6pt);
                    \fill[orange!50] (165:2.5cm) circle(6pt);
                    \fill[orange!50] (375:2.5cm) circle(6pt);
                    \fill[gray!30] (195:2.5cm) circle(6pt);
                    \fill[gray!30] (345:2.5cm) circle(6pt);
                    \fill[olive!50] (225:2.5cm) circle(6pt);
                    \fill[olive!50] (315:2.5cm) circle(6pt);
                    \fill[red!50] (255:2.5cm) circle(6pt);
                    \fill[red!50] (285:2.5cm) circle(6pt);
                \foreach \x in {105,135,...,435}
                {
                \draw[red] (\x:2.5) circle (7pt);
                }
                \draw[shift={(120:2.5cm)},rotate=30,red] (0,0) ellipse(30pt and 13pt);
                \draw[shift={(180:2.5cm)},rotate=90,red] (0,0) ellipse(30pt and 13pt);
                \draw[shift={(240:2.5cm)},rotate=150,red] (0,0) ellipse(30pt and 13pt);
                \draw[shift={(300:2.5cm)},rotate=30,red] (0,0) ellipse(30pt and 13pt);
                \draw[shift={(0:2.5cm)},rotate=90,red] (0,0) ellipse(30pt and 13pt);
                \draw[shift={(60:2.5cm)},rotate=150,red] (0,0) ellipse(30pt and 13pt);
            \end{tikzpicture}
            \subcaption{$\tau\cdot\vx_N^{\va}=\vx_N^{\va}$}
          \end{subfigure}
          \caption{Isotropy groups of $N$ and $\vx_N^{\va}$'s}\label{fig:isotropy_Dih}
        \end{figure}
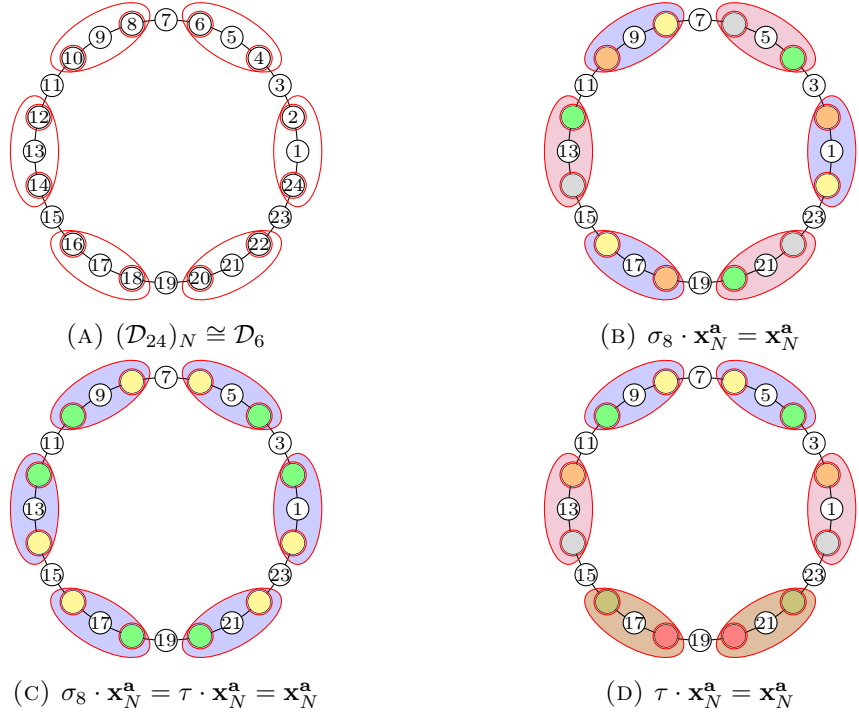
    We set
    \begin{equation*}
        \begin{array}{llll}
            N_1:\qquad&A_{1,1}=\{1,2,3\},&B_{1,1}=\{1\},&B_{1,2}=\{3\}\\
            N_2:\qquad&A_{2,1}=\{5,6,7\},&B_{2,1}=\{5\},&B_{2,2}=\{7\}\\
            N_3:\qquad&A_{3,1}=\{9,10,11\},&B_{3,1}=\{9\},&B_{3,2}=\{11\}\\
            N_4:\qquad&A_{4,1}=\{13,14,15\},&B_{4,1}=\{13\},&B_{4,2}=\{15\}\\
            N_5:\qquad&A_{5,1}=\{17,18,19\},&B_{5,1}=\{17\},&B_{5,2}=\{19\}\\
            N_6:\qquad&A_{6,1}=\{21,22,23\},&B_{6,1}=\{21\},&B_{6,2}=\{23\}
        \end{array}
    \end{equation*}

    \begin{enumerate}
        \item If $a_{1,1}=\cdots=a_{6,1}$ and $b_{1,1}=\cdots=b_{6,1}=b_{2,1}=\cdots=b_{6,2}$, then $(\Dih_{24})_{\vx_N^\va}=(\Dih_{24})_N\cong\Dih_6$.
        \item Assume that $a_{1,1}=a_{3,1}=a_{5,1}$, $a_{2,1}=a_{4,1}=a_{6,1}$, $b_{1,1}=b_{3,1}=b_{5,1}$, $b_{1,2}=b_{3,2}=b_{5,2}$, $b_{2,1}=b_{4,1}=b_{6,1}$, and $b_{2,2}=b_{4,2}=b_{3,2}$. If  $b_{1,1}\neq b_{2,1}$, $b_{1,2}\neq b_{2,2}$, or $a_{1,1}\neq a_{2,1}$, then $L_1=N_1\cup N_2$, $L_2=N_3\cup N_4$, $L_3=N_5\cup N_6$, and $(\Dih_{n+1})_{\vx_{L_1}^\va}=\langle e\rangle$. Hence $(\Dih_{24})_{\vx_N^\va}$ is $\langle \sigma_8\rangle\cong\cC_3$, see Figure~\ref{fig:isotropy_Dih}-(b).
        \item Assume that $a_{1,1}=a_{2,1}=a_{3,1}=a_{4,1}=a_{5,1}=a_{6,1}$ and $b_{1,1}=b_{2,2}=b_{3,1}=b_{4,2}=b_{5,1}=b_{6,2}$ and  $b_{1,2}=b_{2,1}=b_{3,2}=b_{4,1}=b_{5,2}$. If $b_{1,1}\neq b_{1,2}$, then $L_1=N_1\cup N_2$, $L_2=N_3\cup N_4$, $L_3=N_5\cup N_6$, and $(\Dih_{n+1})_{\vx_{L_1}^\va}= \Dih_1$. Hence $(\Dih_{24})_{\vx_N^\va}$ is $\langle\sigma_8,\tau\rangle\cong\Dih_3$, see Figure~\ref{fig:isotropy_Dih}-(c).
        \item Assume that $a_{1,1}=a_{6,1}$, $a_{2,1}=a_{4,1}$, $a_{3,1}=a_{4,1}$ $b_{1,1}=b_{6,2}$, $b_{1,2}=b_{6,1}$, $b_{2,1}=b_{5,2}$, $b_{2,2}=b_{5,1}$, $b_{3,1}=b_{4,2}$, and $b_{3,2}=b_{4,1}$. If $a_{1,1}\neq a_{2,1}$, $a_{2,1}\neq a_{3,1}$, $b_{1,j}\neq b_{2,j}$, or $b_{2,j}\neq b_{3,j}$, then $L_1=N$ and $(\Dih_{24})_{\vx_N^\va}$ is $\langle\tau\rangle\cong\Dih_1$, see Figure~\ref{fig:isotropy_Dih}-(d).
    \end{enumerate}
    \end{example}

    Let us consider the representation of $(\Dih_{n+1})_N$ on $\C[x_I\mid I\in N]\vx_N$. If $(\Dih_{n+1})_N$ is trivial, then $(\Dih_{n+1})_{\vx_N^\va}$ is also trivial, and the Hilbert series of $\C[x_I\mid I\in N]\vx_N$ is $\left(\frac{t}{1-t}\right)^{|N|}.$ If $(\Dih_{n+1})_N=\langle\tau\rangle\cong\Dih_1$, then $(\Dih_{n+1})_{\vx_N^\va}$ is trivial or $\Dih_1$.
    Hence $$\C[x_I\mid I\in N]{\vx_N}=\mathrm{Ind}_{\Dih_1}^{\Dih_1}\C\langle \vx_N^\va\mid (\Dih_{n+1})_{\vx_N^\va}\cong\Dih_1\rangle\oplus \mathrm{Ind}^{\Dih_1}_{\langle e\rangle}\C\langle \vx_N^\va\mid (\Dih_{n+1})_{\vx_N^\va}\cong\langle e\rangle\,\rangle.$$
    Let $|N|=a+2b$, where $a$ is the number of elements $I\in N$ such that $\tau\cdot I=I$. Then the Hilbert series of $\C\langle \vx_N^\va\mid (\Dih_{n+1})_{\vx_N^\va}\cong\Dih_1\rangle$ and $\C\langle \vx_N^\va\mid (\Dih_{n+1})_{\vx_N^\va}\cong\langle e\rangle\,\rangle$ are $$\left(\frac{t}{1-t}\right)^a \left(\frac{t^2}{1-t^2}\right)^{b} \qquad \text{and}\qquad \frac{1}{2}\left\{\left(\frac{t}{1-t}\right)^{a+2b}-\left(\frac{t}{1-t}\right)^{a} \left(\frac{t^2}{1-t^2}\right)^{b}\right\},$$ respectively.
    In general, if $(\Dih_{n+1})_N=H<\Dih_{n+1}$, then
    \begin{equation*}
        \C[x_I\mid I\in N]\vx_N=\bigoplus_{H'<H}\mathrm{Ind}_{H'}^{H}\C\langle\vx_N^\va\mid (H)_{\vx_N^\va}\cong H'\rangle.
    \end{equation*}
    That is, if $(\Dih_{n+1})_N\cong\cC_d$, then we have
    \begin{equation*}
    \C[x_I\mid I\in N]\vx_N =\bigoplus_{\ell\mid d}\mathrm{Ind}_{\cC_{\ell}}^{\cC_{d}}\C\langle\vx_N^\va\mid (\cC_d)_{\vx_N^\va}\cong\cC_{\ell}\rangle;
    \end{equation*}
    if $(\Dih_{n+1})_N\cong\Dih_d$, then we have
    \begin{equation*}
    \C[x_I\mid I\in N]\vx_N =\bigoplus_{\ell\mid d}\left(\mathrm{Ind}_{\Dih_{\ell}}^{\Dih_{d}}\C\langle\vx_N^\va\mid (\Dih_d)_{\vx_N^\va}\cong\Dih_{\ell}\rangle \oplus \mathrm{Ind}_{\cC_{\ell}}^{\Dih_{d}}\C\langle\vx_N^\va\mid (\Dih_d)_{\vx_N^\va}\cong\cC_{\ell}\rangle\right).
    \end{equation*}

    \begin{lemma}\label{lem:hil_cyc_ell}
        When $H:=(\Dih_{n+1})_N$ is isomorphic to $\cC_d$, the Hilbert series of $\C\langle\vx_N^\va\mid (H)_{\vx_N^\va}\cong\cC_{\ell}\rangle$ is
    $$\frac{\ell}{d}\sum_{m\mid\frac{d}{\ell}}\mu(m)\left(\frac{t^{m\ell}}{1-t^{m\ell}}\right)^{\frac{|N|}{m\ell}},$$
    where $\mu$ is the classical M\"obius function of number theory.
    \end{lemma}
    \begin{proof}
        Note that $(H)_{\vx_N^\va}\cong \cC_d$ if and only if $a_{i,j}=a_{i',j}$ for $1\leq i,i'\leq d$. Hence the Hilbert series of $\C\langle\vx_N^\va\mid (H)_{\vx_N^\va}\cong \cC_d\rangle$ is $\left(\frac{t^d}{1-t^d}\right)^{|N|/d}.$

        For two divisors $\ell$ and $\ell'$ of $d$, if $\ell\mid\ell'$, then $a_{i,j}=a_{i',j}$ for $i\equiv i'\pmod {\frac{d}{\ell'}}$ implies that $a_{i,j}=a_{i',j}$ for $i\equiv i'\pmod {\frac{d}{\ell}}$. Hence we need to use the inclusion-exclusion principle to find the Hilbert series of $\C\langle\vx_U^\va\mid (\cC_d)_{\vx_U^\va}\cong\cC_{\ell}\rangle$.

        Note that two divisors $\ell$ and $\ell'$ of $d$ satisfy $\ell\mid\ell'$ if and only if there is an integer $m\mid \frac{d}{\ell}$ such that $m\ell=\ell'$. Hence the inclusion-exclusion principle says that the Hilbert series of $\C\langle\vx_N^\va\mid (\Dih_d)_{\vx_N^\va}\cong\cC_{\ell}\rangle$ is
    $$\frac{\ell}{d}\sum_{m\mid\frac{d}{\ell}}\mu(m)\left(\frac{t^{m\ell}}{1-t^{m\ell}}\right)^{\frac{|N|}{m\ell}}.$$
    \end{proof}

    For example, for the nested set $N=\{1,12,4,45\}$ in Example~\ref{ex:iso_nested}, $(\Dih_6)_N\cong\cC_2$ and the Hilbert series of $\C\langle\vx_N^\va\mid (\cC_2)_{\vx_N^\va}\cong\langle e\rangle\,\rangle$ is $$\frac{1}{2}\left\{\left(\frac{t}{1-t}\right)^4 - \left(\frac{t^2}{1-t^2}\right)^2\right\}.$$

    \begin{lemma}\label{lem:hil_dih_ell}
    When $H:=(\Dih_{n+1})_N$ is isomorphic to $\Dih_d$ for $d<n+1$, the Hilbert series of $\C\langle\vx_N^\va\mid (H)_{\vx_N^\va}\cong\Dih_{\ell}\rangle$ is
    \begin{equation*}
        \frac{\ell}{d}\sum_{m\mid \frac{d}{\ell}}\mu(m)\left(\frac{t^{m\ell}}{1-t^{m\ell}}\right)^{\frac{ad}{m\ell}} \left(\frac{t^{2m\ell}}{1-t^{2m\ell}}\right)^{\frac{bd}{m\ell}},
    \end{equation*} and  the Hilbert series of $\C\langle\vx_N^\va\mid (H)_{\vx_N^\va}\cong\cC_{\ell}\rangle$ is
    $$\frac{\ell}{2d}\sum_{m\mid\frac{d}{\ell}}\left\{ \left(\frac{t^{m\ell}}{1-t^{m\ell}}\right)^{\frac{a+2b}{m\ell}} - \left(\frac{t^{m\ell}}{1-t^{m\ell}}\right)^{\frac{ad}{m\ell}} \left(\frac{t^{2{m\ell}}}{1-t^{2{m\ell}}}\right)^{\frac{bd}{m\ell}}\right\}.$$
    \end{lemma}
    \begin{proof}
        First, $(H)_{\vx_N^\va}\cong \Dih_d$ if and only if $a_{i,j}=a_{i',j}$ and $b_{i,j}=b_{i',j}=b_{i,2b+1-j}=b_{i',2b+1-j}$ for $1\leq i,i'\leq d$. Hence the Hilbert series of $\C\langle\vx_N^\va\mid (H)_{\vx_N^\va}\cong \Dih_d\rangle$ is $\left(\frac{t^d}{1-t^d}\right)^{a}\left(\frac{t^{2d}}{1-t^{2d}}\right)^{b}.$
        From~\eqref{lem:iso_vx_va}, $(H)_{\vx_N^\va}\cong \cC_\ell$ for some $\ell\mid d$ if and only if the exponents $a_{i,j}$'s and $b_{i,k}$'s satisfy \eqref{eq:cond_cyclic} and \eqref{eq:cond_reflect}. Note that if two divisors $\ell$ and $\ell'$ of $d$ satisfy $\ell\mid\ell'$ and the exponents $a_{i,j}$'s and $b_{i,k}$'s satisfy \eqref{eq:cond_cyclic} and \eqref{eq:cond_reflect} with respect to $\ell'$, then the exponents $a_{i,j}$'s and $b_{i,k}$'s satisfy \eqref{eq:cond_cyclic} and \eqref{eq:cond_reflect} with respect to $\ell$. Hence, from the inclusion-exclusion principle, the Hilbert series of $\C\langle\vx_N^\va\mid (H)_{\vx_N^\va}\cong \Dih_d\rangle$ is
        \begin{equation*}
            \frac{\ell}{d}\sum_{m\mid \frac{d}{\ell}}\mu(m)\left(\frac{t^{m\ell}}{1-t^{m\ell}}\right)^{\frac{ad}{m\ell}} \left(\frac{t^{2m\ell}}{1-t^{2m\ell}}\right)^{\frac{bd}{m\ell}}.
        \end{equation*}

        Secondly, the Hilbert series of $\C\langle\vx_N^\va\mid (H)_{\vx_N^\va}\cong \cC_d\rangle$ is $$\frac{1}{2}\left\{\left(\frac{t^d}{1-t^d}\right)^{a+2b} - \left(\frac{t^d}{1-t^d}\right)^{a}\left(\frac{t^{2d}}{1-t^{2d}}\right)^{b}\right\}$$ from the inclusion-exclusion principle. Consequently,  the Hilbert series of $\C\langle\vx_N^\va\mid (H)_{\vx_N^\va}\cong \cC_d\rangle$ is $$\frac{\ell}{2d}\sum_{m\mid\frac{d}{\ell}}\left\{ \left(\frac{t^{m\ell}}{1-t^{m\ell}}\right)^{\frac{a+2b}{m\ell}} - \left(\frac{t^{m\ell}}{1-t^{m\ell}}\right)^{\frac{ad}{m\ell}} \left(\frac{t^{2{m\ell}}}{1-t^{2{m\ell}}}\right)^{\frac{bd}{m\ell}}\right\}.$$
    \end{proof}

    Applying Lemmas~\ref{lem:hil_cyc_ell} and~\ref{lem:hil_dih_ell} to~\eqref{eq:decom2_equiv_coho}, we can conclude the following.
    \begin{proposition}\label{prop:hilbert_vxN}
        For a nested set $N\in\mathcal{N}(C_{n+1})$, if $(\Dih_{n+1})_N\cong\cC_d$, then the Hilbert series of $\C[x_I\mid I\in N]\vx_N$ is
        \begin{equation*}
            \sum_{\ell\mid d} \frac{\ell}{d} (\mathrm{Ind}_{\cC_\ell}^{\cC_d}1) \sum_{m\mid\frac{d}{\ell}}\mu(m)\left(\frac{t^{m\ell}}{1-t^{m\ell}}\right)^{\frac{|N|}{m\ell}},
        \end{equation*}
        if $(\Dih_{n+1})_N\cong\Dih_d$, then the Hilbert series of $\C[x_I\mid I\in N]\vx_N$ is
        \begin{equation*}
        \begin{split}
            &\sum_{\ell\mid d}\frac{\ell}{d} \sum_{m\mid\frac{d}{\ell}} \mu(m)\left[\left( \mathrm{Ind}_{\Dih_\ell}^{\Dih_d}1\right) \left(\frac{t^{m\ell}}{1-t^{m\ell}}\right)^{\frac{ad}{m\ell}} \left(\frac{t^{2m\ell}}{1-t^{2m\ell}}\right)^{\frac{bd}{m\ell}}\right.\\
            &\qquad\qquad\qquad\quad\left. +\frac{1}{2}\left(\mathrm{Ind}_{\cC_\ell}^{\Dih_d}1\right)\left\{ \left(\frac{t^{m\ell}}{1-t^{m\ell}}\right)^{\frac{(a+2b)d}{m\ell}} - \left(\frac{t^{m\ell}}{1-t^{m\ell}}\right)^{\frac{ad}{m\ell}}\left(\frac{t^{2m\ell}}{1-t^{2m\ell}}\right)^{\frac{bd}{m\ell}}\right\} \right].
        \end{split}
        \end{equation*}
    \end{proposition}

\section{Dihedral group representations on $H_T^\ast(M_{C_{n+1}};\C)$}\label{sec:action_equiv_coho}

    In this section, we deduce explicit formulas for the dihedral group representation on $H_T^\ast(M_{C_{n+1}})$.

    Note that from Lemmas~\ref{lem:iso_nested_set} and~\ref{lem:iso_vx_va}, we have
    \begin{equation}\label{eq:decom1_equiv_coho}
    \begin{split}
        H_T^\ast(M_{C_{n+1}};\C)&= \bigoplus_{N\in{\mathcal{N}}(C_{n+1})}\bigoplus_{\va\in(\Z_{>0})^{N}}\C\langle \prod_{I\in N}x_I^{a_I}\rangle,\\
        &=\bigoplus_{H<\Dih_{n+1}}\bigoplus_{N\C[x_I\mid I\in N]\vx_N\atop (\Dih_{n+1})_N\cong H} \mathrm{Ind}_H^{\Dih_{n+1}} \C[x_I\mid I\in N]\vx_N,
    \end{split}
    \end{equation} and we know the representation of $(\Dih_{n+1})_N$ on the subring $\C[x_I\mid I\in N]\vx_N$. Now we are ready to deduce the representation of $\Dih_{n+1}$ on $H_T^\ast(M_{C_{n+1}})$.

    Recall that we define the numbers
    \begin{equation*}
        \begin{split}
            \alpha_{n+1}(d,k)&:=\left|\{N\in \mathcal{N}_k(C_{n+1})\mid (\Dih_{n+1})_N\cong\cC_d\}\right|,\\
            \beta_{n+1}(d,k)&:=\left|\{N\in \mathcal{N}_k(C_{n+1})\mid (\Dih_{n+1})_N\cong\Dih_d\}\right|,\text{ and}\\
            \gamma_{n+1}(d,k)&:=\alpha_{n+1}(d,k)+\beta_{n+1}(d,k).
        \end{split}
    \end{equation*}
    By using the definition of the sets $A_{i,j}$ and $B_{i,j}$ defined in~\eqref{eq:sets_A_B}, we also define
    \begin{equation*}
            \beta_{n+1}(d,k,a):=\left|\{N\in\mathcal{N}_k(C_{n+1})\mid (\Dih_{n+1})_N\cong\Dih_d,\,|A_{i,j}|=a, \text{ and }|B_{i,j}|=k-a\}\right|.
    \end{equation*}
    \begin{theorem}
        The representation of $\Dih_{n+1}$ on $H_T^\ast(M_{C_{n+1}};\C)$ is
        \begin{equation*}
        \begin{split}
        &1+\sum_{k=1}^n\sum_{d\mid n+1\atop d\mid k}\gamma_{n+1}(d,k)\sum_{\ell\mid d}(\mathrm{Ind}_{\cC_\ell}^{\Dih_{n+1}}1)\frac{\ell}{n+1}\sum_{m\mid\frac{d}{\ell}}\mu(m)\left(\frac{t^{m\ell}}{1-t^{m\ell}}\right)^{\frac{k}{m\ell}}\\
        &+\sum_{k=1}^n\sum_{d\mid n+1\atop d\mid k}\sum_{a=1}^k\beta_{n+1}(d,k,a)\sum_{\ell\mid d}(\mathrm{Ind}_{\Dih_\ell}^{\Dih_{n+1}}1)\frac{\ell}{2(n+1)} \sum_{m\mid\frac{d}{\ell}}\mu(m)\left(\frac{t^{m\ell}}{1-t^{m\ell}}\right)^{\frac{ad}{m\ell}}\left(\frac{t^{2m\ell}}{1-t^{2m\ell}}\right)^{\frac{bd}{m\ell}},
        \end{split}
        \end{equation*}where $\mu$ is the classical M\"{o}bius function in number theory.
    \end{theorem}
    \begin{proof}
        We can rewrite~\eqref{eq:decom1_equiv_coho} as follows.
        \begin{equation*}
          \begin{split}
            H_T^\ast(M_{C_{n+1}};\C)&=\left(\bigoplus_{H<\Dih_{n+1}}\bigoplus_{N\in{\mathcal{N}}(C_{n+1}) \atop (\Dih_{n+1})_N\cong H}\bigoplus_{H'<H}\mathrm{Ind}_{H'}^{\Dih_{n+1}}\C\langle\vx_N^\va\mid (H)_{\vx_N^\va}\cong H'\rangle\right).
          \end{split}
        \end{equation*}
        Note that if $N=\emptyset$, then $\vx_N=1$, $(\Dih_{n+1})_N=\Dih_{n+1}$, and $\C[x_I\mid I\in N]\vx_N=\C$. Since $\mathrm{Ind}_{H}^{\Dih_{n+1}}\left(\mathrm{Ind}_{H'}^{H}1\right)=\mathrm{Ind}_{H'}^{H}1$, the theorem follows from the above by applying Proposition~\ref{prop:hilbert_vxN}.
    \end{proof}

    Considering the actions of the cyclic group $\cC_{n+1}$ on $\mathcal{N}(C_{n+1})$ and $H_T^\ast(M_{C_{n+1}};\C)$, we get the following representation.
    \begin{corollary}
        The representation of $\cC_{n+1}$ on $H_T^\ast(M_{C_{n+1}};\C)$ is
        \begin{equation*}
        \begin{split}
        &1+\sum_{k=1}^n\sum_{d\mid n+1\atop d\mid k}\gamma_{n+1}(d,k)\sum_{\ell\mid d}(\mathrm{Ind}_{\cC_\ell}^{\cC_{n+1}}1)\frac{\ell}{n+1}\sum_{m\mid\frac{d}{\ell}}\mu(m)\left(\frac{t^{m\ell}}{1-t^{m\ell}}\right)^{\frac{k}{m\ell}},
        \end{split}
        \end{equation*}where $\mu$ is the classical M\"{o}bius function in number theory.
    \end{corollary}

\section{Relationship with annular non-crossing matchings}\label{sec:annular_noncrossing}

    In this section, we construct annular non-crossing matchings and then find a relationship with nested sets. We also discuss the relationship between the number of annular non-crossing matchings and the number $\beta_{n+1}(d,k)$.

    Note that if $n+1$ is even, then there are two kinds of reflections; one fixes two vertices of $C_{n+1}$ and the other has no fixed vertices. If $n+1$ is odd, then every reflection fixes exactly one vertex. Now we define annular non-crossing matchings related to a nested set $N\in\mathcal{N}(C_{n+1})$ which can be fixed under some reflection in $\Dih_{n+1}$.\footnote{In fact, our annular non-crossing matchings are circular non-crossing matchings, non-crossing matchings of curves embedded within a disk (\cite{GT1990}), but for convenience of explanation we use the idea of the annular non-crossing matchings in~\cite{DP2016}, and they are slightly different from the original definition; we add more conditions.}

    Before we construct annular non-crossing matchings, we first consider an arrangements of beads on a disjoint union of arcs on a unit cycle.

    \medskip

    \noindent{\textbf{Arrangements of beads on a disjoint union of intervals.}} Let $\Gamma_1,\ldots,\Gamma_\ell$ be pairwise disjoint arcs on a unit circle arranged in counterclockwise. We put $k$ beads with colors blue and white on  $\Gamma_1\cup\cdots\cup\Gamma_\ell$ in the following rules.
    \begin{enumerate}
        \item[(R1)] Put a blue (respectively, white) bead $B_1$ on $\Gamma_{i_1}$.
        \item[(R2)] We put a bead $B_2$ on $\Gamma_{i_2}$, depending on its color.
            \begin{enumerate}
                \item[(R2-1)] If both $B_1$ and $B_2$ are blue (respectively, white), then  $i_2<i_1$ (respectively, $i_2>i_1$), and we do not put on the intervals $\Gamma_j$ for $j\geq i_1$ (respectively, $j\leq i_1$) any more.
                \item[(R2-2)] If $B_1$ and $B_2$ have the different colors, that is, $B_2$ is white (respectively, blue), then $i_2>i_1$ (respectively, $i_2<i_1$),
                    and then we choose one of the union of intervals $\Gamma_{i_1}\cup\cdots \cup\Gamma_{i_2-1}$ or $\Gamma_{i_1+1}\cup\cdots\cup\Gamma_{i_2}$ in order not to put any bead on it.
            \end{enumerate}
            \item[(R3)] We put a bead $B_{i_3}$ on some possible interval by comparing with the color of $B_{i_2}$ in the same rule (R2), and continue in this fashion until we arrange $k$ beads.
    \end{enumerate}

    \bigskip

    Consider the annulus $\{(x,y)\in\R^2\mid \frac{1}{4}\leq x^2+y^2\leq 1\}$ with dots $v_{n+1,i}$ on the outer circle, $1\leq i\leq n+1$. Let $\Gamma_i$ be the open arc between $v_{n+1,i}$ and $v_{n+1,i+1}$ for $i=1,\ldots,\lfloor\frac{n+1}{2}\rfloor$.

    \bigskip

    \noindent{\textbf{Construction of annular non-crossing matchings of type 1.}} We put $k$ beads with colors blue and white on the arcs $\Gamma_1\cup\cdots\cup\Gamma_{\lfloor\frac{n+1}{2}\rfloor}$ under the rules above.

    By matching the beads in the following five steps, we obtain an annular non-crossing matching, see Figure~\ref{fig:ann_matching}.
    \begin{enumerate}
        \item[(S1)] Identify all blue bead that lie directly right of a white bead.
        \item[(S2)] Repeatedly apply the previous step.
        \item[(S3)] If there is a blue bead which not lying right of a white bead, draw a line to the arc $\{(x,0)\mid -1\leq x\leq-\frac{1}{4}\}$.
        \item[(S4)] If there is a white bead which not connected to a blue bead, draw a line to the arc $\{(x,0)\mid \frac{1}{4}\leq x\leq 1\}$.
        \item[(S5)] Reflect along the $x$-axis.
    \end{enumerate}

    \bigskip

    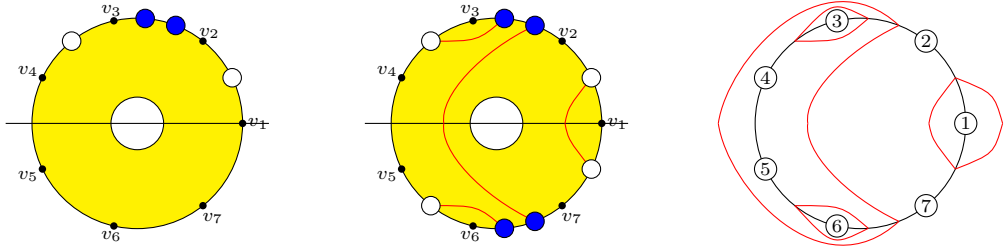
\begin{figure}[h]
    \centering
        \begin{subfigure}{.3\textwidth}
        \centering
            \begin{tikzpicture}[scale=.7]
                \filldraw[fill=yellow] (0,0) circle(2cm);
                \filldraw[fill=white] (0,0) circle(.5cm);
                \draw (-2.5,0)--(2.5,0);
                \fill (2,0) circle(2pt);
                \fill (360/7:2cm) circle(2pt);
                \fill (2*360/7:2cm) circle(2pt);
                \fill (3*360/7:2cm) circle(2pt);
                \fill (4*360/7:2cm) circle(2pt);
                \fill (5*360/7:2cm) circle(2pt);
                \fill (6*360/7:2cm) circle(2pt);
                \draw (2.3,0) node{\tiny$v_1$};
                \draw (360/7:2.2cm) node{\tiny$v_2$};
                \draw (2*360/7:2.2cm) node{\tiny$v_3$};
                \draw (3*360/7:2.3cm) node{\tiny$v_4$};
                \draw (4*360/7:2.3cm) node{\tiny$v_5$};
                \draw (5*360/7:2.2cm) node{\tiny$v_6$};
                \draw (6*360/7:2.3cm) node{\tiny$v_7$};
                \filldraw[fill=white] (360/14:2cm) circle (5pt);
                \filldraw[fill=white] (5*360/14:2cm) circle (5pt);
                \filldraw[fill=blue] (4*360/21:2cm) circle (5pt);
                \filldraw[fill=blue] (5*360/21:2cm) circle (5pt);
            \end{tikzpicture}
        \end{subfigure}
        \begin{subfigure}{.3\textwidth}
        \centering
            \begin{tikzpicture}[scale=.7]
                \filldraw[fill=yellow] (0,0) circle(2cm);
                \filldraw[fill=white] (0,0) circle(.5cm);
                \draw (-2.5,0)--(2.5,0);
                \fill (2,0) circle(2pt);
                \fill (360/7:2cm) circle(2pt);
                \fill (2*360/7:2cm) circle(2pt);
                \fill (3*360/7:2cm) circle(2pt);
                \fill (4*360/7:2cm) circle(2pt);
                \fill (5*360/7:2cm) circle(2pt);
                \fill (6*360/7:2cm) circle(2pt);
                \draw (2.3,0) node{\tiny$v_1$};
                \draw (360/7:2.2cm) node{\tiny$v_2$};
                \draw (2*360/7:2.2cm) node{\tiny$v_3$};
                \draw (3*360/7:2.3cm) node{\tiny$v_4$};
                \draw (4*360/7:2.3cm) node{\tiny$v_5$};
                \draw (5*360/7:2.2cm) node{\tiny$v_6$};
                \draw (6*360/7:2.3cm) node{\tiny$v_7$};
                \draw[red] (5*360/21:2cm)..controls (2*360/7:1.6cm)..(5*360/14:2cm);
                \draw[red] (4*360/21:2cm)..controls (2*360/7:1.4cm) and (3*360/7:1.2)..(-1,0);
                \draw[red] (360/14:2cm)..controls (360/28:1.4cm)..(1.3,0);
                \draw[red] (-360/14:2cm)..controls (-360/28:1.4cm)..(1.3,0);
                \draw[red] (-5*360/21:2cm)..controls (-2*360/7:1.6cm)..(-5*360/14:2cm);
                \draw[red] (-4*360/21:2cm)..controls (-2*360/7:1.4cm) and (-3*360/7:1.2)..(-1,0);
                \filldraw[fill=white] (360/14:2cm) circle (5pt);
                \filldraw[fill=white] (5*360/14:2cm) circle (5pt);
                \filldraw[fill=blue] (4*360/21:2cm) circle (5pt);
                \filldraw[fill=blue] (5*360/21:2cm) circle (5pt);
                \filldraw[fill=white] (-360/14:2cm) circle (5pt);
                \filldraw[fill=white] (-5*360/14:2cm) circle (5pt);
                \filldraw[fill=blue] (-4*360/21:2cm) circle (5pt);
                \filldraw[fill=blue] (-5*360/21:2cm) circle (5pt);
            \end{tikzpicture}
        \end{subfigure}
        \begin{subfigure}{.3\textwidth}
        \centering
            \begin{tikzpicture}[scale=.7]
                \draw (0,0) circle(2cm);
                \foreach \x in {1,2,...,7}
                    {
                    \filldraw[fill=white] (\x*360/7:2cm) circle(6pt);
                    \draw (\x*360/7-360/7:2cm) node {\tiny$\x$};
                    }
                \draw (0:2cm) node {\tiny$1$};
                \draw[red] (5*360/21:2cm)..controls (2*360/7:1.6cm)..(5*360/14:2cm);
                \draw[red] (4*360/21:2cm)..controls (2*360/7:1.4cm) and (3*360/7:1.2)..(-1,0);
                \draw[red] (360/14:2cm)..controls (360/28:1.4cm)..(1.3,0);
                \draw[red] (-360/14:2cm)..controls (-360/28:1.4cm)..(1.3,0);
                \draw[red] (-5*360/21:2cm)..controls (-2*360/7:1.6cm)..(-5*360/14:2cm);
                \draw[red] (-4*360/21:2cm)..controls (-2*360/7:1.4cm) and (-3*360/7:1.2)..(-1,0);
                \draw[red] (5*360/21:2cm)..controls (2*360/7:2.4cm)..(5*360/14:2cm);
                \draw[red] (4*360/21:2cm)..controls (2*360/7:3.3cm) and (3*360/7:2.8)..(-2.7,0);
                \draw[red] (360/14:2cm)..controls (360/28:2.6cm)..(2.7,0);
                \draw[red] (-360/14:2cm)..controls (-360/28:2.6cm)..(2.7,0);
                \draw[red] (-5*360/21:2cm)..controls (-2*360/7:2.4cm)..(-5*360/14:2cm);
                \draw[red] (-4*360/21:2cm)..controls (-2*360/7:3.3cm) and (-3*360/7:2.8)..(-2.7,0);
            \end{tikzpicture}
        \end{subfigure}
        \caption{The annular non-crossing matching of type 1 corresponding to the nested set $\{\{1\},\{3\},\{6\},\{3,4,5,6\}\}\in\mathcal{N}_4(C_7)$}\label{fig:ann_matching}
    \end{figure}

    Let $\mathrm{Ann}_{n+1}(k)$ be the set of all annular non-crossing matchings of type~1. It follows from the construction that for even $n$ we have $
        |\mathrm{Ann}_{n+1}(k)|=|\mathrm{Ann}_{n}(k)|.$
    \begin{proposition}\label{prop:odd_annular}
        Let $n+1$ be odd. For each integer $1\leq k\leq n$, the sum of beta numbers, $\sum_{d\mid k}\beta_{n+1}(d,k)$, is equal to
        \begin{equation*}
                (n+1)\times |\mathrm{Ann}_{n+1}(k)|.
        \end{equation*}
    \end{proposition}
    \begin{proof}
         For each annular non-crossing matching $\mathcal{M}\in\mathrm{Ann}(n+1,k)$, we can find a nested set $N_{\mathcal{M}}\in \mathcal{N}_k(C_{n+1})$ which is fixed under the reflection $\sigma_1\tau$. Note that $\sigma_i^{-1}\sigma_{2i+1}\tau\sigma_i=\sigma_1\tau$. Hence, for each   $N\in \mathcal{N}(C_{n+1})$, there is an annular non-crossing matching $\mathcal{M}\in\mathrm{Ann}(n+1,k)$ such that $N=\sigma_{i}\cdot N_{\mathcal{M}}$, where $N$ is fixed under the reflection $\sigma_{2i+1}\tau$.
    \end{proof}

    Given an annular non-crossing matching $\mathcal{M}$, let $b$ be the number of pairs of beads directly connected to each other in Steps (S1) and (S2), and let $a=k-2b$. Then we can see that
    \begin{align*}
        a=|\{I\in N_{\mathcal{M}}\mid \sigma_1\tau\cdot I=I\}|\text{ and }2b=|\{I\in N_{\mathcal{M}}\mid \sigma_1\tau\cdot I\neq I\}|.
    \end{align*}
    We denote by $\mathrm{Ann}_{n+1}(k,b)$ the set of annular non-crossing matchings in $\mathrm{Ann}_{n+1}(k)$ such that there are $b$ pairs of beads connecting to each other directly in Steps 1 and 2. Then $\mathrm{Ann}_{n+1}(k)=\cdotcup_{b=1}^k\mathrm{Ann}_{n+1}(k,b).$

    When $n+1$ is even, we also consider another kind of construction of annular non-crossing matchings.

    \bigskip

    \noindent{\textbf{Construction  of annular non-crossing matchings of type 2.}} We put $k$ beads on the arcs $\Gamma_1\cup\cdots\cup\Gamma_{\frac{n-1}{2}}$ satisfying rules (R1)$\sim$(R3), and we slightly change steps (S3)$\sim$(S5) as follows.
    \begin{enumerate}
        \item[(S3${}^\prime$)] If there is a blue bead which not lying right of a white bead, draw a line to the left arc of the intersection of the annulus and the straight line through the origin with angle $-\frac{\pi}{n+1}$.
        \item[(S4${}^\prime$)] If there is a white bead which not connected to a blue bead, draw a line to the right arc of the intersection of the annulus and the straight line through the origin with angle $-\frac{\pi}{n+1}$.
        \item[(S5${}^\prime$)] Reflect along the straight line through the origin with angle $-\frac{\pi}{n+1}$.
    \end{enumerate}

    \bigskip

    Then we obtain an annular non-crossing matching corresponding to a nested set $N\in \mathcal{N}_k(C_{n+1})$ which is fixed under the reflection $\tau$, see Figure~\ref{fig:ann_matching2}.

    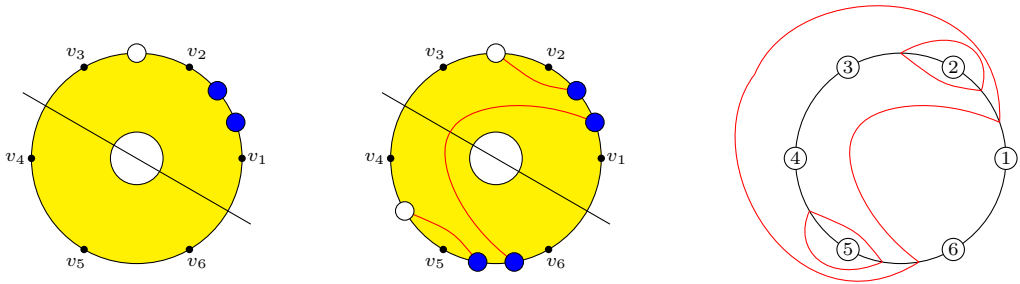
\begin{figure}[h]
    \centering
        \begin{subfigure}{.3\textwidth}
        \centering
            \begin{tikzpicture}[scale=.7]
                \filldraw[fill=yellow] (0,0) circle(2cm);
                \filldraw[fill=white] (0,0) circle(.5cm);
                \draw (-30:2.5cm)--(150:2.5);
                \fill (2,0) circle(2pt);
                \fill (360/6:2cm) circle(2pt);
                \fill (2*360/6:2cm) circle(2pt);
                \fill (3*360/6:2cm) circle(2pt);
                \fill (4*360/6:2cm) circle(2pt);
                \fill (5*360/6:2cm) circle(2pt);
                \draw (2.3,0) node{\tiny$v_1$};
                \draw (360/6:2.3cm) node{\tiny$v_2$};
                \draw (2*360/6:2.3cm) node{\tiny$v_3$};
                \draw (3*360/6:2.3cm) node{\tiny$v_4$};
                \draw (4*360/6:2.3cm) node{\tiny$v_5$};
                \draw (5*360/6:2.3cm) node{\tiny$v_6$};
                \filldraw[fill=white] (90:2cm) circle (5pt);
                \filldraw[fill=blue] (40:2cm) circle (5pt);
                \filldraw[fill=blue] (20:2cm) circle (5pt);
            \end{tikzpicture}
        \end{subfigure}
        \begin{subfigure}{.3\textwidth}
        \centering
            \begin{tikzpicture}[scale=.7]
                \filldraw[fill=yellow] (0,0) circle(2cm);
                \filldraw[fill=white] (0,0) circle(.5cm);
                \draw (-30:2.5cm)--(150:2.5);
                \fill (2,0) circle(2pt);
                \fill (360/6:2cm) circle(2pt);
                \fill (2*360/6:2cm) circle(2pt);
                \fill (3*360/6:2cm) circle(2pt);
                \fill (4*360/6:2cm) circle(2pt);
                \fill (5*360/6:2cm) circle(2pt);
                \draw (2.3,0) node{\tiny$v_1$};
                \draw (360/6:2.3cm) node{\tiny$v_2$};
                \draw (2*360/6:2.3cm) node{\tiny$v_3$};
                \draw (3*360/6:2.3cm) node{\tiny$v_4$};
                \draw (4*360/6:2.3cm) node{\tiny$v_5$};
                \draw (5*360/6:2.3cm) node{\tiny$v_6$};
                \draw[red] (40:2cm)..controls (60:1.6cm)..(90:2cm);
                \draw[red] (20:2cm)..controls (60:1.4cm) and (120:1.2)..(150:1);
                \draw[red] (260:2cm)..controls (240:1.6cm)..(210:2cm);
                \draw[red] (280:2cm)..controls (240:1.4cm) and (180:1.2)..(150:1);
                \filldraw[fill=white] (90:2cm) circle (5pt);
                \filldraw[fill=blue] (40:2cm) circle (5pt);
                \filldraw[fill=blue] (20:2cm) circle (5pt);
                \filldraw[fill=white] (210:2cm) circle (5pt);
                \filldraw[fill=blue] (260:2cm) circle (5pt);
                \filldraw[fill=blue] (280:2cm) circle (5pt);
            \end{tikzpicture}
        \end{subfigure}
        \begin{subfigure}{.3\textwidth}
        \centering
            \begin{tikzpicture}[scale=.7]
                \draw (0,0) circle(2cm);
                \foreach \x in {1,2,...,6}
                    {
                    \filldraw[fill=white] (\x*360/6:2cm) circle(6pt);
                    \draw (\x*60-60:2cm) node {\tiny$\x$};
                    }
                \draw (2,0) node{\tiny$1$};
                \draw[red] (40:2cm)..controls (60:1.6cm)..(90:2cm);
                \draw[red] (20:2cm)..controls (60:1.4cm) and (120:1.2)..(150:1);
                \draw[red] (260:2cm)..controls (240:1.6cm)..(210:2cm);
                \draw[red] (280:2cm)..controls (240:1.4cm) and (180:1.2)..(150:1);
                \draw[red] (40:2cm)..controls (50:2.8cm) and (70:2.7cm)..(90:2cm);
                \draw[red] (20:2cm)..controls (60:4cm) and (120:4cm)..(150:3.2cm);
                \draw[red] (260:2cm)..controls (240:2.8cm) and (220:2.7cm)..(210:2cm);
                \draw[red] (280:2cm)..controls (240:4cm) and (180:4cm)..(150:3.2cm);
            \end{tikzpicture}
        \end{subfigure}
        \caption{The annular non-crossing matching of type 2 corresponding to the nested set $\{\{2\},\{5\},\{2,3,4,5\}\}\in\mathcal{N}_3(C_6)$}\label{fig:ann_matching2}
    \end{figure}

    Let $\mathrm{Ann}'_{n+1}(k)$ be the set of all annular non-crossing matchings of type~2. Since both $\mathrm{Ann}_{n+1}(k)$ and $\mathrm{Ann}_{n}(k)$ are constructed from the same arrangements of beads, we can see that for odd $n$
    \begin{equation}\label{eq:relation_Ann_AnnPrime}
        |\mathrm{Ann}'_{n+1}(k)|=|\mathrm{Ann}_{n}(k)|.
    \end{equation}

    \begin{proposition}
        Let $n+1$ be even. For an integer $1\leq k\leq n$, the sum of beta numbers, $\sum_{d\mid k}\beta_{n+1}(d,k)$, is equal to the number
        \begin{equation*}
                \frac{n+1}{2}\left(|\mathrm{Ann}_{n+1}(k)|+|\mathrm{Ann}_{n}(k))|\right).
        \end{equation*}
    \end{proposition}
    \begin{proof}
        If an annular non-crossing matching $\mathcal{M}$ is in $\mathrm{Ann}(n+1,k)$ (respectively, $\mathrm{Ann}'(n+1,k)$), then we reflect the matching along the $y$-axis (respectively, the straight line through the origin with angle $\frac{(n-1)\pi}{2(n+1)}$), and then we change blues beads to white beads and white beads to blue beads. Let $\mathcal{M}'$ be the resulting annular non-crossing matching. Then the nested set $N_{\mathcal{M}'}$ can be obtained from $N_{\mathcal{M}}$ by the reflection $\sigma_{\frac{n+3}{2}}\tau$ (respectively, $\sigma_{\frac{n+1}{2}}\tau$).

        Since there is no nested set $N$ such that $\tau\cdot N=\sigma_1\tau\cdot N=N$, we get $\mathrm{Ann}(n+1,k)\cap \mathrm{Ann}'(n+1,k)=\emptyset$. From the similar argument to the proof of Proposition~\ref{prop:odd_annular}, we have
        \begin{equation*}
                \frac{n+1}{2}\left(|\mathrm{Ann}_{n+1}(k)|+|\mathrm{Ann}'_{n+1}(k)|\right).
        \end{equation*}
        Therefore, the proposition follows from~\eqref{eq:relation_Ann_AnnPrime}.
    \end{proof}

    Recall that $\beta_{n+1}(d,k)$ is the cardinality of the set
    $$\left\{N\in\mathcal{N}_k(C_{n+1})\mid (\Dih_{n+1})_N\cong\Dih_d\right\}.$$ Hence we can compute $\beta_{n+1}(d,k)$ by using the same argument in the proof of Lemma~\ref{lem:same_size_nested_set}, if we know $|\mathrm{Ann}_{n+1}(k)|$ for odd $n+1$.
    Furthermore, if we know $\mathrm{Ann}_{n+1}(k,a)$, then we can also count the nested sets $N\in\mathcal{N}_k(C_{n+1})$ such that $(\Dih_{n+1})\cong\Dih_d$ and the number of elements $I\in N$ fixed under some reflection $\tau'\in(\Dih_{n+1})_N$ is equal to $a$.
    We can explicitly compute $|\mathrm{Ann}_{n+1}(k,a)|$ when $k$ or $k-a$ is small as follows.

\renewcommand{\arraystretch}{1.15}
\begin{table}[h]
\centering\footnotesize{
\begin{tabular}{c||c | c | c | c }
    \toprule
    \backslashbox{$k$}{$a$} & $k$ & $k-2$ & $k-4$ & $\cdots$\\
    \midrule
    $1$ & $2{\lfloor \frac{n+1}{2}\rfloor\choose 1}$ &&&\\
    $2$ & $3{\lfloor \frac{n+1}{2}\rfloor\choose 2}$ & ${\lfloor \frac{n+1}{2}\rfloor\choose 2}$ &&\\
    $3$ & $4{\lfloor \frac{n+1}{2}\rfloor\choose 3}$ & $4{\lfloor \frac{n+1}{2}\rfloor\choose 3}+2{\lfloor \frac{n+1}{2}\rfloor\choose 2}$ &&\\
    $4$ & $5{\lfloor \frac{n+1}{2}\rfloor\choose 4}$ & $9{\lfloor \frac{n+1}{2}\rfloor\choose 3}+6{\lfloor \frac{n+1}{2}\rfloor\choose 3}$ & $2{\lfloor \frac{n+1}{2}\rfloor\choose 4}+2{\lfloor \frac{n+1}{2}\rfloor\choose 3}$ &\\
    $5$ & $6{\lfloor \frac{n+1}{2}\rfloor\choose 5}$ & $16{\lfloor \frac{n+1}{2}\rfloor\choose 5}+ 12{\lfloor \frac{n+1}{2}\rfloor\choose 4}$ & $10{\lfloor \frac{n+1}{2}\rfloor\choose 5} + 14{\lfloor \frac{n+1}{2}\rfloor\choose 4} + 3{\lfloor \frac{n+1}{2}\rfloor\choose3}$ &\\
    $\vdots$ & $\vdots$ & $\vdots$ & $\vdots$ & \\
    $k$ & $(k+1){\lfloor \frac{n+1}{2}\rfloor\choose k}$& $(k-1)^2{\lfloor \frac{n+1}{2}\rfloor\choose k}+(k-1)(k-2){\lfloor \frac{n+1}{2}\rfloor\choose k-1}$ &  & \\
    \bottomrule
    \end{tabular}}
    \caption{$\mathrm{Ann}_{n+1}(k,a)$}\label{table:all_type}
\end{table}

Note that the coefficients of ${\lfloor \frac{n+1}{2}\rfloor\choose \ell}$ on Table~\ref{table:all_type} for $1\leq\ell\leq k$ are coming from the number of arrangements of two letters $B$ and $W$ satisfying certain conditions, for example, $k+1=\sum_{i=0}^{k}1$ is the number of arrangements satisfying that there is no $W$ lying on the left side of $B$, $(k-1)^2=\sum_{i=0}^{k-2}\left\{\frac{(k-1-i)!}{1!(k-2-i)!}+i\right\}$ is the number of arrangements satisfying that there is only one $W$ lying on the left side of some $B$ or there is only one $B$ lying on the right side of some $W$, and $(k-1)(k-2)=\sum_{i=0}^{k-3}\left\{\frac{(i+1)!}{1!i!}+\frac{(k-2-i)!}{1!(k-3-i)!}\right\}$ is the number of arrangements satisfying that it contains the word $WBB$ or $WWB$ and also satisfying that there is only one $W$ lying on the left side of some $B$ or there is only one $B$ lying on the right side of $W$.

\begin{example}
    For the cycle graph $C_4$, the representation of $\Dih_4$ on $H_T^\ast(M_{C_{n+1}};\C)$ is
    \begin{equation*}
        \begin{split}
            &1 + (\mathrm{Ind}_{\langle e\rangle}^{\Dih_{4}}1)\left[\frac{3t^2}{(1-t)^2} + \frac{2t^3}{(1-t)^3} + \frac{1}{2}\left\{\frac{t^2}{(1-t)^2}-\frac{t^2}{1-t^2}\right\}\right]\\
            &+ (\mathrm{Ind}_{\Dih_1}^{\Dih_{4}}1)\left[\frac{3t}{1-t} + \frac{t^2}{(1-t)^2} +\frac{t}{1-t}\frac{t^2}{1-t^2} +\frac{1}{2} \left\{\frac{t^3}{(1-t)^3}-\frac{t}{1-t}\frac{t^2}{1-t^2}\right\}\right]\\
            &+(\mathrm{Ind}_{\Dih_2}^{\Dih_{4}}1)\left(\frac{t^2}{1-t^2}\right).
        \end{split}
    \end{equation*}
    By substituting $\frac{|G|}{|H|}$ instead of $(\mathrm{Ind}_{H}^{G}1)$ in the above, we get the Poincar\'{e} series of $H_T^\ast(M_{C_4};\C)$ $$1+\frac{12t}{1-t}+\frac{30t^2}{(1-t)^2}+\frac{20t^3}{(1-t)^3}.$$
\end{example}

\section*{Acknowledgement}
The author thanks Professor Mikiya Masuda for his valuable comments and encouragement. The author also thanks Miho Hatanaka, Tatsuya Horiguchi, and Jang Soo Kim.

\end{document}